\documentclass[12pt]{amsart}

\usepackage{amsmath,amsfonts,amssymb, bm}

\usepackage{graphicx,epic,eepic,a4wide,comment}

\newtheorem{lemma}{Lemma}[section]

\newtheorem{prop}[lemma]{Proposition}
\newtheorem{thm}[lemma]{Theorem}
\newtheorem{cor}[lemma]{Corollary}
\theoremstyle{definition}
\newtheorem{defn}[lemma]{Definition}
\newtheorem{example}[lemma]{Example}

\theoremstyle{remark}
\newtheorem{rmk}[lemma]{Remark}

\newtheorem{nota}[lemma]{Notation}

\newcommand{\J}{\mathcal{J}}
\newcommand{\T}{\mathcal{T}}
\newcommand{\C}{\mathcal{C}}
\newcommand{\zmax}{\text{0-max}}
\newcommand{\omin}{\text{1-min}}
\newcommand{\BbZ}{\mathbb{Z}}
\newcommand{\dimh}{\mathrm{dim}_{\tt H}}
\renewcommand{\t}{{\bm t}}

\renewcommand{\epsilon}{\varepsilon}

\numberwithin{equation}{section}
\numberwithin{table}{section}

\begin{document}

\title[The baker's map with a convex hole]{The baker's map with a convex hole}
\author{Lyndsey Clark, Kevin G. Hare and Nikita Sidorov}
\address{}
\email{lyndsey.r.clark@gmail.com}
\address{Department of Pure Mathematics \\
University of Waterloo \\
Waterloo, Ontario \\
Canada N2L 3G1}
\thanks{Research of K. G. Hare was supported by NSERC Grant RGPIN-2014-03154.
Computational support was provided by the Canada Foundation for Innovation.}
\email{kghare@uwaterloo.ca}
\address{
School of Mathematics, The University of Manchester,
Oxford Road, Manchester M13 9PL, United Kingdom.}
\email{sidorov@manchester.ac.uk}

\date{\today}
\begin{abstract}
We consider the baker's map $B$ on the unit square $X$ and an open convex set $H\subset X$
which we regard as a hole. The survivor set $\J(H)$ is
defined as the set of all points in $X$ whose $B$-trajectories are disjoint from $H$.
The main purpose of this paper is to study holes $H$ for which $\dimh \J(H)=0$
(dimension traps) as well as those for which any periodic trajectory of $B$
intersects $\overline H$ (cycle traps).

We show that any $H$ which lies in the interior of $X$ is not a dimension trap. This means that, unlike
the doubling map and other one-dimensional examples, we can have $\dimh \J(H)>0$ for $H$ whose Lebesgue
measure is arbitrarily close to one. Also, we
describe holes which are dimension or cycle traps, critical in the sense that if we
consider a strictly convex subset, then the corresponding property in question no longer holds.

We also determine $\delta>0$ such that $\dimh \J(H)>0$ for all convex $H$ whose Lebesgue
measure is less than $\delta$.

This paper may be seen as a first extension of our work begun in \cite{Lyn,Lyn-thesis,GS15,HS14,Sid14}
to higher dimensions.
\end{abstract}

\maketitle

\section{The doubling map with a hole: recap}
\label{sec:doubling}

Open dynamical systems, i.e. the induced maps on the set of points which do not fall into a certain
predetermined set (a {\em hole}) under iteration by a map, have been studied in numerous papers.
The first discussion was in \cite{PY1979} where they considered an idealized game of billiards without friction, where there is a hole somewhere in the table.
Depending upon the location or size of the hole, the rate of which random billiard balls will ``escape'' was considered.
In \cite{BY2011} it was shown that not only does the size of the hole matter, but that the location of the hole can have huge impact upon the escape rate.
The escape rate in this case can be well approximated by the set of periodic
    points of the map that lie within the hole, \cite{BJP2014}.
For a good history of these problems, see \cite{Demers2006, Demers2014}.
These ideas have also be considered in theoretical computer science,
    in the context of combinatorics on words, in the study of
    {\em unavoidable} words, as first discussed in \cite{Schutz64}.

The set-up is as follows: let $T:X\to X$ be a map and $H$ be an open set.
We define the {\em survivor set for $H$} as follows:
\[
\J(H) := \{x : T^n(x) \not\in H\ \text{for all}\ n \in \BbZ\}
\]
if $T$ is invertible and
\[
\J(H) := \{x : T^n(x) \not\in H\ \text{for all}\ n\ge0\}
\]
otherwise.
If $T$ is invertible, or we wish to emphasis that we are only considering the
    forward orbits, we will use the notation $\J^+(H)$.
In \cite{GS15} P.~Glendinning and the third author investigated
the case $X=[0,1]$ and the doubling map $Tx=2x\bmod1$, with $H=(a,b)$ with
$0<a<b<1$. Here is a brief summary of the main results of that paper.

We need some definitions and basic results from combinatorics on words, which we borrow
from \cite{GS15}.  See also see \cite[Chapter~2]{Loth} for a detailed exposition.
For any two finite 0-1 words $u=u_1\dots u_k$ and $v=v_1\dots v_n$ we write
$uv$ for their concatenation $u_1\dots u_k v_1\dots v_n$.
We write $u^n = \underbrace{u u \dots u}_{n}$ as the concatenation of $u$ with
    itself $n$ times, and $u^\infty$ as the infinite concatenation of $u$
    with itself.
Let $w$ be a finite or infinite word. We say that a finite  word $u$ is a
{\em factor of} $w$ if there exists $k$ such that $u=w_k\dots w_{k+n}$ for some $n\ge0$.
For a finite word $w$ let $|w|$ stand for its length and $|w|_1$ stand for the number of
1s in $w$. The 1-{\em ratio} of $w$ is defined as $|w|_1/|w|$.
For an infinite word $w_1w_2\dots$ the 1-ratio is defined as $\lim_{n\to\infty}|w_1\dots w_n|_1/n$, if it exists.

We say that a finite or infinite word $w$ is {\em balanced} if for any $n\ge1$ and any two factors
$u,v$ of $w$ of length~$n$ we have $||u|_1-|v|_1|\le1$. An infinite word is called {\em Sturmian}
if it is balanced and not eventually periodic. A finite word $w$ is {\em cyclically balanced} if $w^2$ is balanced.
(And therefore, $w^\infty$ is balanced.) It is well known that if $u$ and $v$ are two cyclically balanced words
with $|u|=|v|=q$ and $|u|_1=|v|_1=p$ and $\gcd(p,q)=1$, then $u$ is a cyclic permutation of $v$.
Thus, there are only $q$ distinct cyclically balanced words of length $q$ with $p$ 1s.

We say that a finite or infinite word $u$ is {\em lexicographically smaller than} a word $v$
(notation: $u\prec v$) if either $u_1<v_1$ or there exists an $n\ge1$ such that $u_i = v_i$ for
$i=1,\dots, n$ and $u_{n+1}<v_{n+1}$.

For any $r=p/q\in\mathbb Q\cap(0,1)$ we define the substitution $\rho_r$ on $\{0,1\}$ as follows:
$\rho_r(0)=\zmax(r)$, the lexicographically largest cyclically balanced word of length~$q$ with
1-ratio $r$ beginning with 0, and $\rho_r(1)=\omin(r)$, the lexicographically smallest cyclically
balanced word of length~$q$ with 1-ratio $r$ beginning with 1. In particular, $\rho_{1/2}(0)=01,\ \rho_{1/2}(1)=10$.

Consider some $(a,b)\subset (0,1)$ such that $\dimh \J^+(a,b)=0$. We then begin to move $a$ and $b$ towards
each other until we get a survivor set of positive dimension. In doing so, we pass through a number of
milestones, where we gain extra periodic orbits for $T$ which lie outside the interval.
Namely, let $\bm r=(r_1,r_2, \dots)$ be a finite or
infinite vector with each component $r_i\in\mathbb Q\cap(0,1)$.
We define the sequences of 0-1 words parameterized by $\bm r$ as follows:
\begin{align*}
s_n &= \rho_{r_1}\dots \rho_{r_n}(0),\\
t_n &= \rho_{r_1}\dots \rho_{r_n}(1).
\end{align*}
In particular, if $r_i=1/2$ for all $i\le n$, we have that $s_n$ is a truncated Thue-Morse word and
$t_n$ is its mirror image (see Section~\ref{sec:55} for a precise definition of the Thue-Morse word). Put
\begin{align*}
\mathfrak s(\bm r)&=\lim_{n\to\infty}\rho_{r_1}\dots\rho_{r_n}(0),\\
\mathfrak t(\bm r)&=\lim_{n\to\infty}\rho_{r_1}\dots\rho_{r_n}(1).
\end{align*}

\begin{thm} \cite[Theorem~2.13]{GS15}
We have $\dimh \J^+ (\mathfrak s(\bm r), \mathfrak t(\bm r))=0$ and moreover,
$\dimh \J^+ (a,b)>0$ for any $a>\mathfrak s(\bm r), b<\mathfrak t(\bm r)$.
\end{thm}

In other words, any infinite sequence of rationals from $(0,1)$ induces a route to chaos;
moreover, it was shown in \cite[Lemma~2.12]{GS15} that
$\J^+(s_n^\infty, t_n^\infty)\setminus \J^+(s_{n-1}^\infty, t_{n-1}^\infty)$
contains just one cycle, namely, $s_n^\infty$ (which is the same as $t_n^\infty$).

The reason behind this approach is that we need $\J^+(H)$ to be uncountable (and possibly,
of positive Hausdorff dimension) for the induced map $T|_{\J^+(H)}$ to be of interest.

As a corollary of several results from \cite{GS15}, we have the following result,
which, roughly speaking, implies that if a connected hole $H$ for the doubling map
is too small, then its survivor set has positive Hausdorff dimension, whereas
if it is too large, then it intersects all orbits.

\begin{thm}[\cite{GS15}]\label{thm:doubling}
Let $(a,b) \subset [0,1]$.
\begin{enumerate}
\item
We always have $\dimh \J^+(a,b)>0$ if
$$
b-a<\frac14\cdot\prod_{n=1}^{\infty}\left(1-2^{-2^n}\right) \approx0.175092,
$$
and this bound is sharp.
\item If $b-a>\frac12$, then $\J^+(a,b)=\{0,1\}$.
\end{enumerate}
\end{thm}

The purpose of the present paper is to extend this approach to the baker's map. As we will
see, it works perfectly in the symmetric case $r_i\equiv1/2$ (Section~\ref{sec:55}) and
fails for an asymmetric case (Section~\ref{ssec:dim trap55b}). Having said that, the baker's
map generates a great deal of completely new effects in comparison with the doubling map
(see, e.g., Section~\ref{sec:01}).

\section{Baker's map: preliminaries}

\label{sec:intro}
Put $X=[0,1]^2$. The baker's map $B:X\to X$ is the natural extension of the doubling map, conjugate to the
shift map on the set of bi-infinite sequences. Namely,
 \begin{equation*}
  B(x,y) = \begin{cases}
            (2x, \frac{y}{2}) &\text{ if } 0 \leq x < \frac{1}{2}, \\
            (2x-1, \frac{y+1}{2}) &\text{ if } \frac{1}{2} \leq x < 1.
           \end{cases}
 \end{equation*}
Put
\begin{align*}
R_0 &= \{ (x,y) : 0 \leq x < 1/2 \}, \\
R_1 &= \{ (x,y) : 1/2 < x \le 1 \}.
\end{align*}
We associate each point $(x,y) \in [0,1) \times [0,1)$ with a bi-infinite sequence
$(x_n)_{-\infty}^\infty \in \{0, 1\}^{\BbZ}$ as follows:
\begin{equation*}
 x_n = k \iff B^{-n}(x,y) \in R_k.
\end{equation*}
(We cannot say what happens when $x=\frac12$.)
The map $\pi$ conjugating the left shift $\sigma:\{0,1\}^{\mathbb Z}\to \{0,1\}^{\mathbb Z}$ and $B$ is given by the formula
\begin{equation}\label{eq:pi}
\pi((x_n)_{-\infty}^\infty)=\left(\sum_{n=1}^\infty x_n2^{-n}, \sum_{n=0}^\infty x_{-n}2^{-n-1}\right).
\end{equation}
In other words, if we write $x$ and $y$ in base 2
    as $x = 0.x_1 x_2 x_3 \ldots$ and $y = 0.y_1 y_2 y_3 \ldots$, then
\begin{align*}
B(x,y) &= (0.x_2 x_3 x_4 \ldots, 0.x_1 y_1 y_2 y_3 \ldots),\\
     B^{-1}(x,y) &= (0.y_1 x_1 x_2 x_3 x_4 \ldots, 0. y_2 y_3 y_4 \ldots).
\end{align*}
We will denote the associated bi-infinite sequence as $\ldots y_3 y_2 y_1 \cdot x_1 x_2 x_3 \ldots$
We thus have
\begin{align*}
\sigma(\ldots y_3 y_2 y_1 \cdot x_1 x_2 x_3 \ldots) & =
       \ldots y_3 y_2 y_1 x_1 \cdot x_2 x_3 \ldots  \\
\sigma^{-1}(\ldots y_3 y_2 y_1 \cdot x_1 x_2 x_3 \ldots) & =
       \ldots y_3 y_2 \cdot y_1 x_1 x_2 x_3 \ldots
\end{align*}
We will use these two notations interchangeably.
Let $H\subset X$; following Section~\ref{sec:doubling}, we define the survivor set for $H$ as follows:
\[
\J(H) := \{(x,y) : B^n(x,y) \not\in H\ \text{for all}\ n \in \BbZ\}.
\]
Let now $[u_{-k}\ldots u_0\cdot u_1\ldots u_m]$ denote the set of all
$(x_n)_{-\infty}^\infty\in\{0,1\}^{\mathbb Z}$ such that $x_i=u_i$ for $-k\le i\le m$
(a {\em cylinder}). See Figure~\ref{fig:cylinders}
for the case $m=k=1$. 

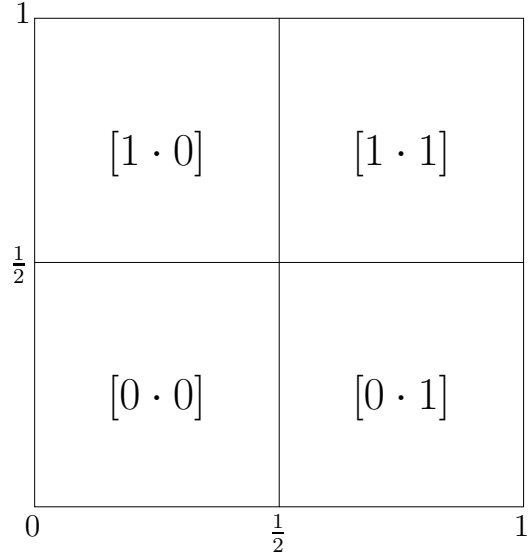
\begin{figure}[t]
\centering
\centering \unitlength=1.3mm
    \begin{picture}(50,53)(0,0)
            \thinlines
  \path(0,0)(0,50)(50,50)(50,0)(0,0)
  \path(0,25)(50,25)
  \path(25,0)(25,50)
  \put(-1,-3){$0$}
    \put(49,-3){$1$}
    \put(-2,49.5){$1$}
  \put(-2.5,24){$\frac12$}
    \put(24,-3.5){$\frac12$}
    {\Large \put(7.5,35){$[1\cdot0]$}
            \put(7.5,10){$[0\cdot0]$}
            \put(32.5,35){$[1\cdot1]$}
            \put(32.5,10){$[0\cdot1]$}
    }
          \end{picture}

\bigskip
\caption{Cylinders for the baker's map} \label{fig:cylinders}
\end{figure}

\begin{defn}
\begin{itemize}
\item We will say that an open set $H$ is a {\em complete trap} if $\J(H)$ does not contain any points except,
possibly, of the form $\pi(\dots 11110000\dots)$ or $\pi(\dots 00001111\dots)$, both
orbits lying on the boundary of $X$.
\item We will say that an open set $H$ is a {\em cycle trap} if $\J(\overline H)$ does not contain any cycles of $B$.
\item We will say that an open set $H$ is a {\em dimension trap} if $\dimh \J(H) = 0$.
\end{itemize}
\end{defn}

It is clear that any complete trap is both a cycle trap and a dimension trap.
In this paper we wish to study those $H$ such that $H$ is a cycle trap (resp. dimension trap),
and $H$ is \textbf{convex}.
The only exception will be considered in Appendix.
As we see below, if we do not restrict our holes to convex sets, then it is possible to
    have arbitrarily small complete traps.

\begin{thm}
\label{thm:epsilon}
For all $\epsilon > 0$ exists a connected non-convex complete trap $H$
    such that its Lebesgue measure $\mathcal L$ is less than $\epsilon$.
\end{thm}

\begin{proof}
Let $A_1$ be the set $(1/2, 1) \times (0, 1/2)$. Clearly, the only sequences
that avoid $A_1$ are of the form $\dots 11110000\dots$, which implies that $A_1$
is a complete trap.

Now we choose $n$ such that $A_1$ can be completely tiled by $1/2^n \times 1/2^n$
    cylinders
    $C_i = [b^{i}_n b^{i}_{n-1} \dots b^{i}_1 \cdot a^{i}_1 a^{i}_2 \dots a^{i}_n]$.
We see that there are $2^{2n} \mathcal L(A_1)$ cylinders $C_i$.
By choosing $n$ large enough, we make sure that the number of such cylinders
    is even.

Now we divide this set of cylinders into two sets of equal size, say $D_1$ and $D_2$.
Let $A_2 := B^n(D_1) \cup B^{-n}(D_2)$.
We claim that $A_2$ is a connected complete trap, and that
    $\mathcal L(A_2) = \mathcal L(A_1) - \mathcal L(A_1)^2/4$.

To see that this is a complete trap, we observe if some element in the orbit
    of $(x,y)$, say $(x', y')$ is in $A_1$, then we have that either
    $B^{-n}(x', y') \in D_1$ or $B^{n} (x', y') \in D_2$.
Therefore, the orbit of $(x,y)$ falls into $A_2$.

We observe that all images in $B^n(D_1)$ are rectangles with
     width $1$ and height $1/2^{2n}$.
That is, they are horizontal strips going from $x = 0$ to $x = 1$.
We further observe that all images in $B^{-n}(D_2)$ are rectangles with
    height $1$ and width $1/2^{2n}$.
That is, they are vertical strips going from $y = 0$ to $y = 1$.
Clearly each horizontal strip overlaps each vertical strip, hence
    this is a connected set.

Moreover, each horizontal strip overlaps each vertical strip in a block of
    size $1/2^{2n} \times 1/2^{2n}$.
There are $\#(D_1) \cdot \#(D_2) =
    \left(2^{2n-1} \mathcal L(A_1)\right) \cdot \left(2^{2n-1} \mathcal L(A_1) \right)
    = 2^{4n-2} \mathcal L(A_1)^2$ such overlaps.
Thus, the size of $A_2$ is the size of the vertical strips
    (which is $\mathcal L(A_1) /2$),
    plus the size of the horizontal strips (which is $\mathcal L(A_1) / 2$),
    minus the size of the overlap (which is $2^{4n-2} \mathcal L(A_1)^2 /2^{4n} =
    \mathcal L(A_1)^2 / 4$).
Thus, $\mathcal L(A_2) = \mathcal L(A_1) - \mathcal L(A_1)^2/4$.

We can repeat this process to get $A_k$, with $\mathcal L(A_k) =
    \mathcal L(A_{k-1}) -\mathcal L(A_{k-1})^2/4$.
The fact that $\mathcal L(A_k)\to0$ as $k\to\infty$ completes the proof.
\end{proof}

\begin{rmk}
The existence of arbitrarily small complete traps for the shift has been
known since the seminal paper \cite{Schutz64}. The novelty of our Theorem~\ref{thm:epsilon}
is the fact that our hole is connected. 
\end{rmk}

See Figure~\ref{fig:epsilon} for $A_1$ and a possible choice for $A_2$ and $A_3$
    from Theorem~\ref{thm:epsilon}.

\begin{figure}
\includegraphics[width=120pt,height=120pt,angle=270]{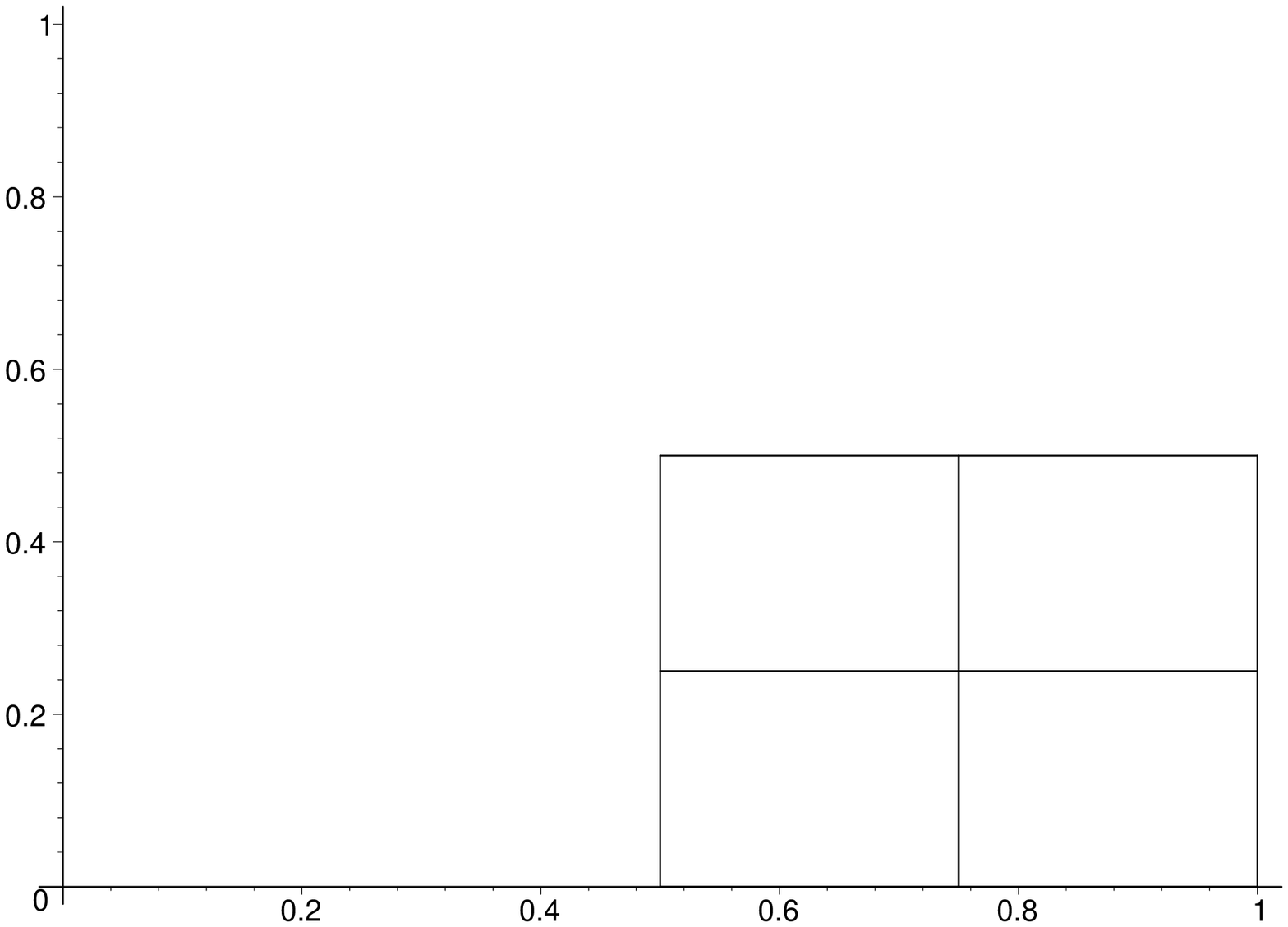}
\includegraphics[width=120pt,height=120pt,angle=270]{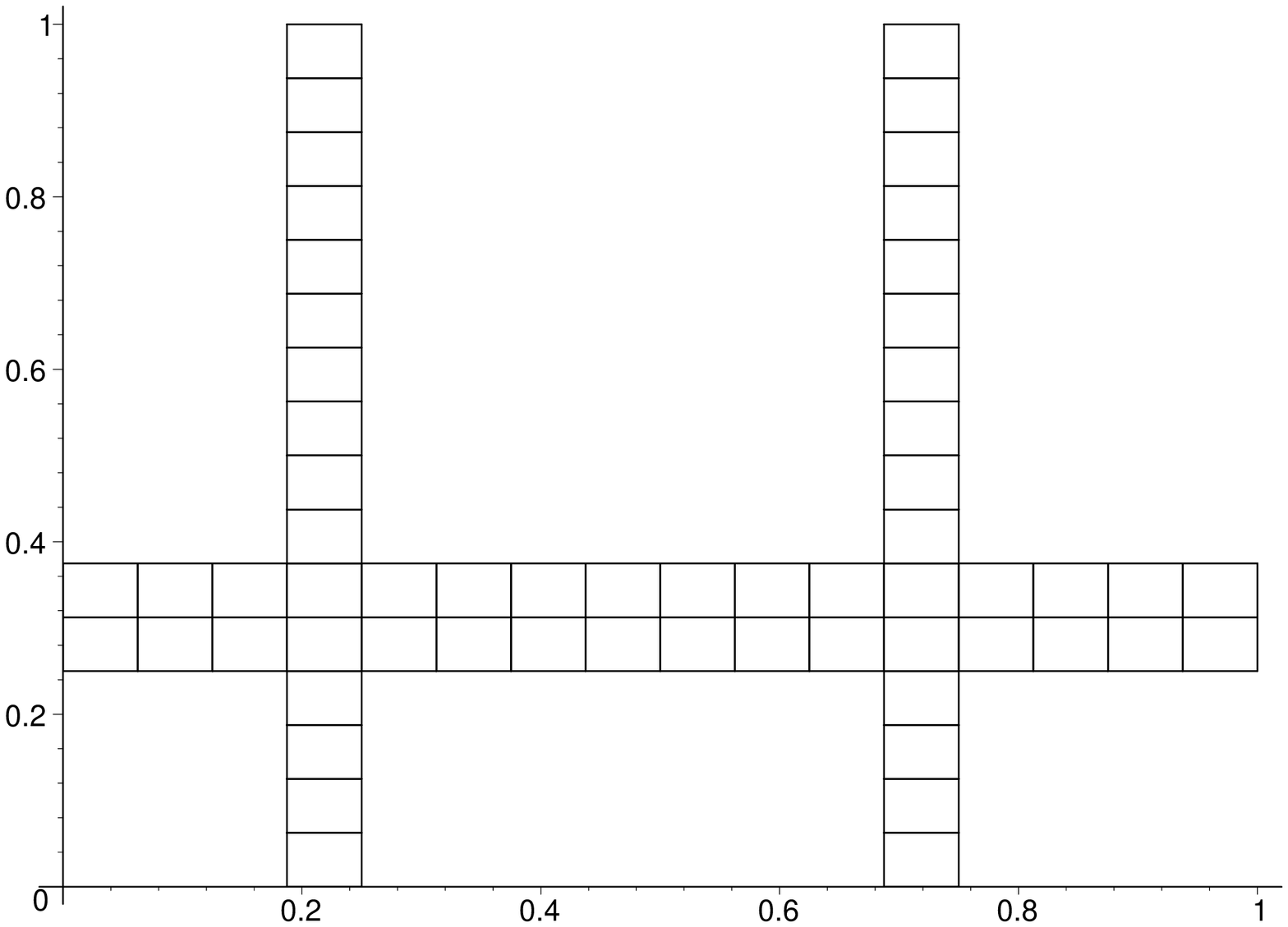}
\includegraphics[width=120pt,height=120pt,angle=270]{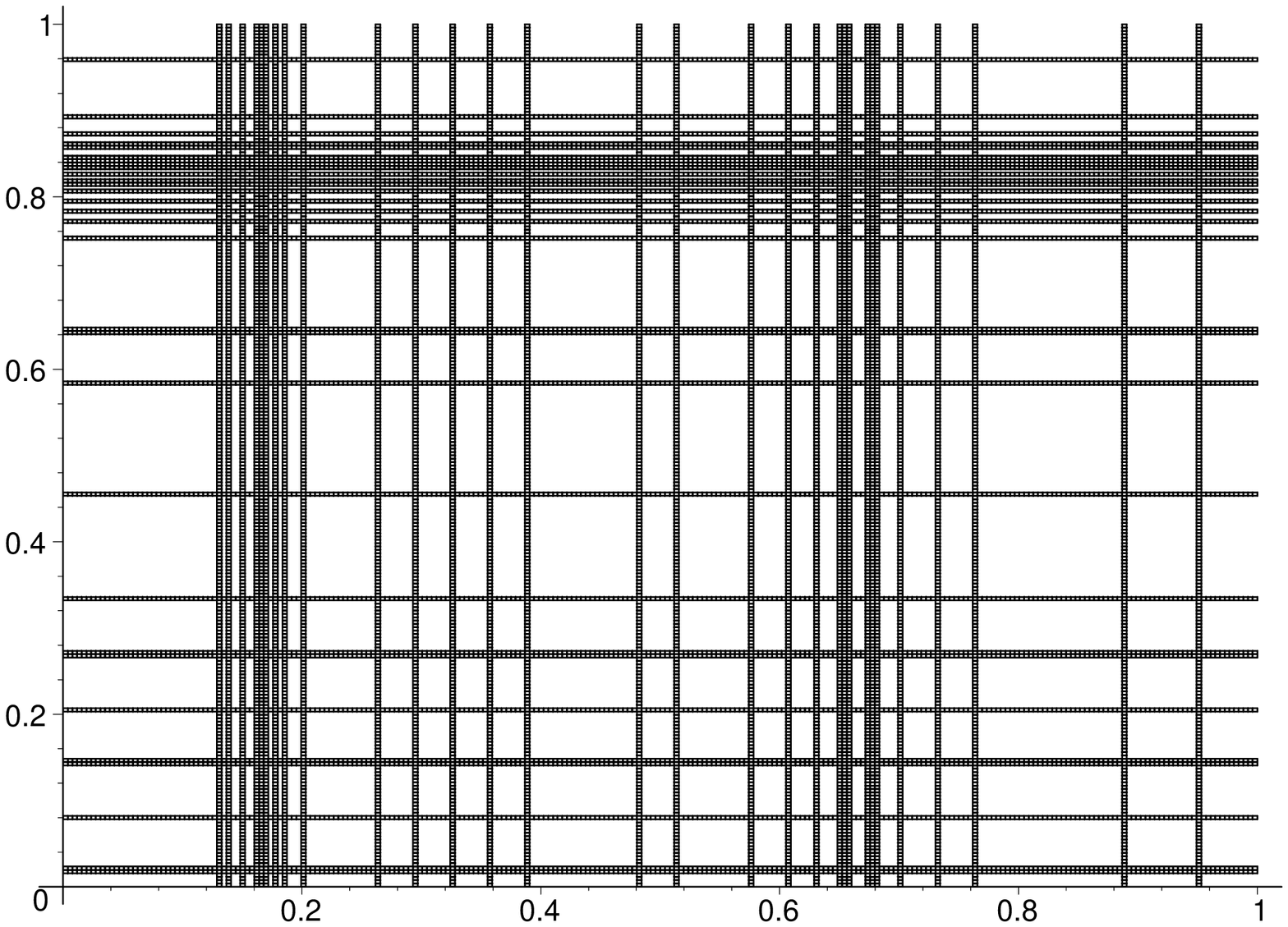}
\caption{Possible $A_1, A_2$ and $A_3$ from Theorem \ref{thm:epsilon}}
\label{fig:epsilon}
\end{figure}

Henceforward we assume $H$ to be convex.
Our first easy observation is that $\left\{\left(\frac13, \frac23), (\frac23, \frac13\right)\right\}$
is the 2-cycle for the baker's map.
As bi-infinite sequences, these correspond to the cycle of $\sigma$ given by
    \[\{ \ldots 0101\cdot 0101 \ldots\ ,\  \ldots 1010\cdot 1010 \ldots \}.\]
Hence if $H$ is a cycle trap, then $\overline H$ must contain at least one of these two points.
The next result shows that if $H$ is a dimension trap, then it must contain
    a part of the boundary of $X$.
This helps to motivate the shapes of $H$ that we investigate.
We recall that $A \subset \{0, 1\}^\BbZ$ is called a {\em subshift} if
    it is closed under the product topology, and invariant under the
    shift operator.
We first prove the claim which
relates the topological entropy of a subshift to the Hausdorff dimension of its image.

\begin{lemma}\label{lem:entropy-dim}
Let $Y$ be a subshift of $\{0,1\}^{\mathbb Z}$ and let $h(Y)$ be its topological entropy ($\log$ base~$2$).
Then
\[
\dimh(\pi(Y))=2h(Y).
\]
In particular, $h(Y)=0\iff \dimh(\pi(Y))=0$.
\end{lemma}
\begin{proof}\footnote{This proof was given by Anthony Quas on MathOverFlow \cite{Quas}, for
which the authors are grateful to him.} Let
$\mu$ stand for a measure of maximal entropy for $Y$. Then by the Shannon-McMillan-Breiman
theorem, for any $\varepsilon>0$, we have a set $Y'\subset Y$ of full $\mu$-measure and all
$x=(x_n)_{n=-\infty}^\infty\in Y'$,
\[
2^{-2k(h-\epsilon)}\le\mu [x_{-k+1}\ldots x_0\cdot x_1\ldots x_k]
 \le 2^{-2k(h+\epsilon)},\quad k\ge n(x).
\]
Let $N_0$ be large enough that $\mu (S)>1/2$, where $S=\{x\in Y' : n(x)\le N_0\}$.

The set $\pi([x_{-k+1}\ldots x_0\cdot x_1\ldots x_k])$ is a $2^{-k}\times 2^{-k}$ square, so its
$s$-Hausdorff measure $\mathcal H^s$ is $2^{-sk}$. Hence
\[
2^{2k(h-\epsilon)}\cdot 2^{-ks}\le \mathcal H^s(\pi(S)) \le
2^{2k(h+\epsilon)}\cdot 2^{-ks}.
\]
Now the claim follows from the definition of the Hausdorff dimension.
\end{proof}

\begin{thm}
\label{thm:An}
 Suppose $H$ is an interior hole; that is to say $\overline{H} \subset (0,1) \times (0,1)$.
 Then $\dimh \J(H) > 0$.
\end{thm}
\begin{proof}
Define $A_n$ as the Cantor set $\{1 0^n, 1 0^{n+1}\}^*$.
That is
\begin{equation*}
 A_n := \left\{ (u_k)_{k=-\infty}^\infty : (u_k)\
 \text{ consists of concatenations of}\ 10^n\ \text{and}\ 10^{n+1} \right\}. \\
\end{equation*}
We claim that for every interior hole $H$, there exists an $n$ such that
 $A_n$ avoids $H$.
Since the topological entropy of $A_n$ is clearly positive for each $n$,
Lemma~\ref{lem:entropy-dim} will then yield the desired result.

Since $H$ is an interior hole, there exists $\epsilon>0$ such that if $(x,y) \in H$ then $x, y > \epsilon$.
Pick $n$ such that $0.0^{\lfloor n/2 \rfloor} 1 (0^n 1)^\infty < \epsilon$.
This implies that if $(x,y) \in A_n$ then for all $i \in \BbZ$ we have that
    $(x_i, y_i) = B^i(x,y)$ has either $x_i < \epsilon$ or $y_i < \epsilon$.
Hence the orbit of $B^i(x,y)$ avoids $H$ as required, which proves the result.
\end{proof}

\begin{rmk}
This means that, in a startling contrast to the one-dimensional case (see
Theorem~\ref{thm:doubling}), there are holes
of size arbitrarily close to full measure which are avoided by a set of positive Hausdorff dimension
and also arbitrarily small connected holes whose survivor set contains only
the two orbits on the boundary of $X$.
\end{rmk}

\begin{rmk}
\label{rmk:required points}
By considering the Cantor set $\{1 0^n, 10^{n+1}\}^*$ we see from the above
      proof that we must contain a point of the form $\{(0, 1/2^n), (1/2^n, 0)\}$ for some $n \geq 1$.
Similarly by considering the Cantor set $\{0 1^n, 01^{n+1}\}^*$ we see that we must contain
      a point of the form $\{(1, 1-1/2^n), (1-1/2^n, 1)\}$ for some $n \geq 1$.
We show in Section \ref{sec:comp} that if a convex set $H$ is a dimension trap
    such that $H$ is closed under the symmetry $(x,y) \leftrightarrow (1-y, 1-x)$ then
    either $\mathcal L(H) > 0.1381$ or $H$ contains the points $(0,1/2)$ and $(1/2, 1)$ or it
    contains the points $(1/2, 0)$ and $(1, 1/2)$.
This motivates the restriction considered in Section \ref{sec:01}.

Similarly if a convex set $H$ is a dimension trap
    such that $H$ is closed under the symmetry $(x,y) \leftrightarrow (1-x, 1-y)$ then
    either $\mathcal L(H) > 0.13$ or $H$ contains the points $(0,1/2)$ and $(1, 1/2)$ or it
    contains the points $(1/2, 0)$ and $(1/2,1)$.
This motivates the restriction considered in Section \ref{sec:55}.
\end{rmk}

In view of the symmetry with respect to $y=x$, we
divide this paper into two main sections.
In Section~\ref{sec:01}, we consider dimension traps that
    contain $(0,1/2)$ and $(1/2,1)$.
We show that the open convex polygon~$\Delta$ with
   corners $(0,1), (1/2, 1), (1/3,2/3)$ and  $(0,1/2)$
   is both a cycle trap and a dimension trap.
We also give two polygons contained in $\Delta$ such that
   they are dimension traps that are in some sense optimal.

In Section~\ref{sec:55}, we consider traps and dimension traps that
    contain $(1/2,1/2)$.
We show that the open parallelogram~$P$ with
   vertices $(1/3,2/3), (1/2,1), (2/3,1/3)$ and $(1/2,0)$ is a cycle trap and
   describe all cycles for which there is an element lying on $\partial P$.

We then proceed to construct a sequence of nested polygons,
   the first of which is $P$ where each  of these polygons is a dimension trap.
The limiting polygon is the hexagon with vertices
    $(1/2, 0), (\t, 2\t), (\t, 2-4\t), (1/2, 1), (1-\t, 2-4\t), (1-4\t, 2\t)$ with
    $\t$ being the Thue-Morse constant.

We also consider an asymmetrical family of
    parallelograms with vertices $(1/2, 0), (a, 2a),\linebreak (1/2, 1), (b, 2b-1)$,
    where the $a$ and $b$ are parameterized by the rationals $r\in(0,1)$. For all of these we prove
    that they are neither cycle nor dimension traps -- with the exception of $r=\frac12$, which corresponds to $P$.
    This shows that, unlike the doubling map (see Section~\ref{sec:doubling}), there appears
    to be no natural asymmetric ``route to chaos'' for the baker's map.

Finally, in Section~\ref{sec:comp} we find $\delta>0$ such that any hole whose
measure is less than $\delta$ is not a dimension trap (Theorem~\ref{thm:delta}).

\section{Traps which contain $(0,1/2)$ and $(1/2,1)$.}
\label{sec:01}

\subsection{Symmetric cycle trap that contains $(0,1)$.}
\label{ssec:trap01}
{\ }

Denote by $\Delta$ the open convex polygon with corners $(0,1), (1/2, 1), (1/3,2/3)$
and  $(0,1/2)$. The results about this trap also hold for its reflection in the line
$y=x$ by swapping maxima and minima as needed.
One of these two traps is depicted in Figure~\ref{convex}.

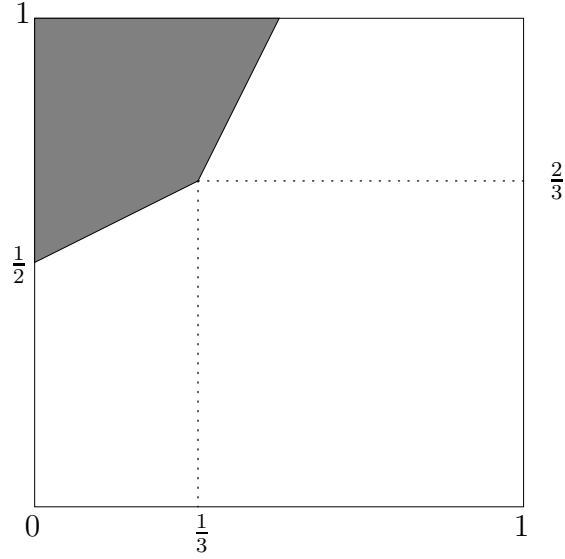
\begin{figure}[t]
\centering
\centering \unitlength=1.3mm
    \begin{picture}(50,53)(0,0)
            \thinlines
  \path(0,0)(0,50)(50,50)(50,0)(0,0)
    \shade\path(0,25)(0,50)(25,50)(16.7,33.35)(0,25)
  \dottedline(50,33.35)(16.7,33.35)
  \dottedline(16.7,33.35)(16.7,0)
    \put(-1,-3){$0$}
    \put(49,-3){$1$}
    \put(-2,49.5){$1$}
    \put(16.2,-3.5){$\frac13$}
            \put(52.5,32.6){$\frac23$}
            \put(-2.5,24){$\frac12$}
          \end{picture}

\bigskip
\caption{The cycle trap $\Delta$} \label{convex}
\end{figure}

Let
\begin{align*}
\widetilde\Omega =
&     \{ \ldots 0^{a_2} 1 0^{a_1} 1 0^{a_0} \cdot 1^{b_0} 0 1^{b_1} 0 1^{b_2} \ldots
        \mid 1 \leq a_i \leq a_{i+1}, 1 \leq b_i \leq b_{i+1}\}
          \cup \\
    &\{\dots 1^{b_0} 0 1^{b_0} 0  \cdot 1^{b_0} 0 1^{b_1} 0 1^{b_2} \ldots
        \mid 1 \leq b_i \leq b_{i+1}\}
          \cup \\
    &\{ \ldots 0^{a_2} 1 0^{a_1} 1 0^{a_0} \cdot 1 0^{a_0} 1 0^{a_0} \ldots
        \mid 1 \leq a_i \leq a_{i+1}\}.
\end{align*}
Here we allow $a_i = \infty$ or $b_j = \infty$ or both.
Let \[ \Omega = \cup_{k \in \BbZ} \sigma^k(\widetilde\Omega).\]

\begin{lemma}
\label{lem:Omega}
The orbit of a point does not fall into $\Delta$ if and only if the point
    is in $\Omega$.
\end{lemma}

\begin{proof}
We first observe that $[1 1 \cdot 0 0] \subseteq \Delta$.
This implies that if a point in $\J(\Delta)$ has a subsequence $1^n 0$ with $n \geq 2$ then this subsequence must be
    of the form $1^n 0 1^m$.
Further more as $\ldots 1^{m+1} \cdot 0 1^{m} 0 \ldots$ falls into $\Delta$ we see that if
    $\ldots 01^n\cdot 01^m0\ldots \not\in\Delta$ then $n \leq m$.
Similarly if we have a subsequence $1 0^n$ we must have a subsequence
    $0^m 1 0^n$ with $n \leq m$.
Thus the only points that avoid $\Delta$ are points of the form
\begin{align*}
\ldots 0^{a_2} 1 0^{a_1} 1 0^{a_0} 1^{b_0} 0 1^{b_1} 0 1^{b_2} \ldots, \\
\ldots 1^{b_0} 0 1^{b_0} 0 1^{b_0} 0 1^{b_0} 0 1^{b_1} 0 1^{b_2} \ldots,\ \  \mathrm{or} \\
\ldots 0^{a_2} 1 0^{a_1} 1 0^{a_0} 1 0^{a_0} 1 0^{a_0} 1 0^{a_0} \ldots,
\end{align*}
where $1 \leq a_1 \leq a_2 \leq a_3 \leq \dots$ and $1 \leq b_1 \leq b_2 \leq b_3 \dots$.

Conversely, let a point be in $\Omega$. To prove that its orbit is disjoint from $\Delta$,
it suffices to check the intersection of the orbit with the cylinder $[1\cdot0]$
    -- see Figure~\ref{fig:cylinders}.
If we have the bi-infinite sequence is of the form
    $\ldots 0^{a_k}1\cdot 0^{a_{k-1}}1\ldots$ with $a_k \geq a_{k-1}$, then
    \[ y = 0.1 0^{a_k} 1 \ldots \leq 0.1 0^{a_{k-1}} 1 \ldots = \frac{1+x}{2} \]
from which it follows that $(x,y) \not \in \Delta$.
If instead the bi-infinite sequence is of the form
    $\ldots 1^{b_k} \cdot 0 1^{b_{k+1}} 0 \ldots$ then
    \[ y = 0.1^{b_k} 0 \ldots \leq 1^{b_{k+1}} 0 \ldots = 2 x. \]
Again, it follows then that $(x,y) \not \in \Delta$.
\end{proof}

\begin{cor}
If a point does not fall into $\Delta$ then the corresponding bi-infinite 0-1 sequence
is either eventually periodic both to the left and to the right
(of which there are only countably many) or comes arbitrarily close to $(0,1/2)$ or $(1/2,1)$.
\end{cor}
\begin{proof}
If our sequence is not eventually periodic to the left or the right, then we have shifts of
    one of the two forms
\begin{align*}
\ldots 10^{a_n} 1 &\cdot 0^{a_{n-1}} 1 \ldots, \\
 \ldots 01^{b_{m-1}} &\cdot 01^{b_m} 0 \ldots,
\end{align*}
with $a_n, b_m$ arbitrarily large.
\end{proof}

\begin{cor}
The convex polygon $\Delta$ is a cycle trap.
\end{cor}

\begin{proof}
The only cycles in $\Omega$ are of the form $(0^n 1)^\infty$ or $(1^n 0)^\infty$ for some $n$.
All of these cycles intersect $\overline \Delta$.
\end{proof}

The next result shows that $\Delta$ is a dimension trap.
This theorem is improved later by Theorem \ref{thm:Delta_0}.
It is included here because it introduces a key idea that is
    used in later proofs.

\begin{thm}\label{thm:Gamma-dimtrap}
The convex polygon $\Delta$ is a dimension trap.
\end{thm}

\begin{proof}
By Lemmas~\ref{lem:Omega} and \ref{lem:entropy-dim}, it suffices to
    show that $\dimh(\Omega) = 0$.
Let $\Omega_n = \{(x_k)_{k=-n}^n : (x_k)_{k=-\infty}^\infty \in \Omega \}$.
We will show that
    $\lim \frac{1}{n} \log\#\Omega_n = 0$, which will prove the result.

Consider a point in $\Omega_n$.
We see that we can write this as
   \[
   0^s 1 0^{a_\ell} 1 0^{a_{\ell-1}} 1 \ldots 1 0^{a_0} 1^{b_0} 0 1^{b_1} 0 \ldots 0 1^{b_{m}} 0 1^{t},
   \]
where $1 \leq a_0 \leq a_1 \leq \dots \leq a_\ell$,
      $1 \leq b_0 \leq b_1 \leq \dots \leq b_m$,
      $0 \leq s, t \leq n$.
Here there may be no $a_i$ or no $b_i$, giving an empty sum.
Let
\begin{align*}
k & = \left| 0^s 1 0^{a_\ell} 1 0^{a_{\ell-1}} 1 \ldots 1 0^{a_0} \right| \\
  & = \left( \sum_{i=0}^\ell a_i +1 \right) + s
\end{align*}
and similarly
\begin{align*}
2n-k & = \left|1^{b_0} 0 1^{b_1} 0 \ldots 0 1^{b_{m}} 0 1^{t} \right| \\
  & = \left( \sum_{i=0}^m b_i +1 \right) + t
\end{align*}
For fixed $k$, we see that the number of such points is bounded above by the number
    of partitions of $k-s$ and $2n-k-t$ respectively.
Denote the number of partitions of $i$ by $p(i)$.
This gives that (with $p(i)\equiv0$ for $i\le0$)
\[
\# \Omega_n  \leq \sum_{k=0}^{2n} \sum_{s,t=0}^n p(k-s) p(2n - k - t)
               \leq (2n+1) n^2 p(2n)^2.
\]
By the Hardy-Ramanujan formula, we have that
    $p(n) \sim \frac{1}{n} c_1 \mathrm{e}^{c_2 \sqrt{n}}$ as $n\to\infty$
    for some some positive constants $c_1$ and $c_2$.
Thus, $\frac{1}{n} \log \#\Omega_n \leq c_3 \frac{\log n}{\sqrt{n}}$ for
    some constant $c_3$, from which the result follows.
\end{proof}

\subsection{Symmetric and asymmetric dimension traps containing $(0,1)$.}
\label{ssec:dim trap01}
{\ }

Let $\Delta' \subset \Delta$ be the quadrilateral  with vertices
    \[
    (0, 1/2), (5/24, 2/3), (1/2, 1), (0, 1)
    \]
and $\Delta'' \subset \Delta$ the symmetric image of $\Delta'$
under $(x,y)\to(1-y, 1-x)$ with vertices
    \[
    (0, 1/2), (1/3, 19/24), (1/2, 1), (0, 1)
    \]
-- see Figure~\ref{convex2} for $\Delta'$.
\begin{figure}[t]
\centering
\centering \unitlength=1.3mm
  \begin{picture}(50,53)(0,0)
    \thinlines
    \path(0,0)(0,50)(50,50)(50,0)(0,0)
    \shade\path(0,25)(0,50)(25,50)(10.42,33.35)(0,25)
    \dottedline(10.42,33.35)(10.42,0)
    \dottedline(10.42,33.35)(50, 33.35)
    \put(-1,-3){$0$}
    \put(49,-3){$1$}
    \put(-2,49.5){$1$}
    \put(9.2,-3.5){$\frac{5}{24}$}
    \put(52.5,32.85){$\frac{2}{3}$}
    \put(-2.5,24){$\frac12$}
  \end{picture}
\bigskip
\caption{The convex dimension trap $\Delta'$} \label{convex2}
\end{figure}
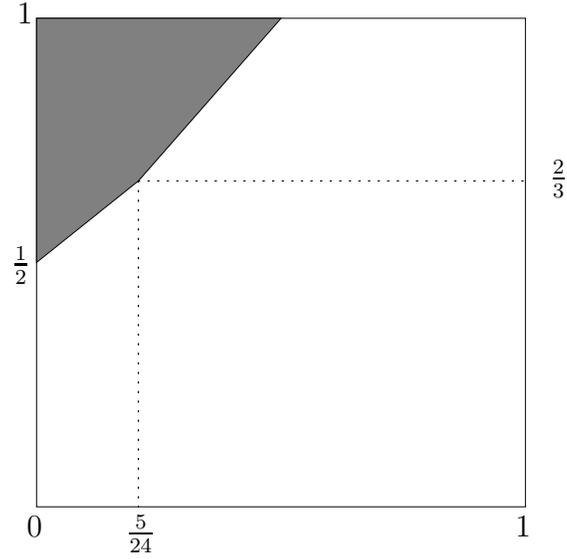
We claim that these two quadrilaterals are minimal convex dimension traps in the
    sense that if a set
    $\Delta_0 \subset \Delta$ is a dimension trap, then $\Delta_0$
    must contain the vertices $(0,1/2), (1/2, 1), (0, 1)$ and
    at least one of the vertices $(5/24, 2/3)$ or $(1/3, 19/24)$.

\begin{thm}\label{thm:Delta_0}
Let $\Delta_0 \subset \Delta$ be a dimension trap, then
either $\Delta' \subset \Delta_0$ or $\Delta'' \subset \Delta_0$.
\end{thm}

Let $\mathrm{hull}(A)$ denote the convex hull of $A\subset\mathbb R^2$.

\begin{cor}
Let $\Delta_0 \subset \Delta$ be a dimension trap symmetric with respect to $x+y=1$.
Then $\mathrm{hull}(\Delta' \cup \Delta'') \subset \Delta_0$.
\end{cor}

\begin{proof}[Proof of Theorem \ref{thm:Delta_0}]
The fact that $\{(0,1/2), (1/2,1)\}\subset\Delta_0$ is contained in the proof
of Theorem~\ref{thm:An} and Remark~\ref{rmk:required points}.
Similarly, we can show that $(0,1)\in\Delta_0$ by considering the Cantor set
    $\{0^n 1^n, 0^{n+1} 1^{n+1}\}^*$.

To see that one of $(5/24,2/3)$ or $(1/3,19/24)$ are in $\Delta_0$, consider the Cantor set
    given by
\[
\{0011 (01)^n, 0011 (01)^{n+1}\}^*.
\]
Let $\epsilon > 0$; we see for $n$ sufficiently large that the cylinder
    $[(01)^n \cdot 0011 (01)^{n}] \subseteq
   (5/24, 2/3)+B_\epsilon$, where $B_\epsilon$ is the disc centred at the origin of radius $\epsilon$.
Furthermore, for $n$ sufficiently large we have
    $[(01)^n 0 0 1 1 \cdot (01)^{n}] \subseteq (1/3, 19/24)+B_\epsilon$.
Hence the Cantor set in question intersects a neighbourhood of these vertices.
Moreover, for $n$ sufficiently large, we see that either
    $x \approx \frac{5}{24}$ for $x > \frac{1}{3} - \epsilon$.
Similarly either $y \approx 19/24$ or $y < 2/3 + \epsilon$.
This shows that a convex subset of $\Delta$ that is a dimension
    trap must contain either $(1/3, 2/3)$ or
    $(5/24, 2/3)$ or $(1/3, 19/24)$.
If this set contains $(1/3, 2/3)$, then it has to be $\Delta$, and by convexity
    it would include the other two points.
Otherwise, it has to contain one of the points $(5/24, 2/3)$ or $(1/3, 19/24)$, and we are done.
\end{proof}

\begin{thm}\label{thm:Delta-prime}
The convex quadrilaterals $\Delta'$ and $\Delta''$ are dimension traps.
\end{thm}

\begin{proof}
We will show that the projection of $\J(\Delta')$ and $\J(\Delta'')$ onto the
    $x$-axis both have dimension $0$.
We then notice that under the symmetric $(x,y)\to (1-y, 1-x)$ that $\Delta' \leftrightarrow \Delta''$.
From this we conclude that the projection of $\J(\Delta')$ and $\J(\Delta'')$ onto the
    $y$-axis both have dimension $0$.
This will prove that both $\Delta'$ and $\Delta''$ are dimension traps as required.

\medskip\noindent
{\bf Projection of $\Delta''$ on the $x$-axis.}

\medskip\noindent
As with $\Delta$, we see that $[1 1 \cdot 0 0] \subseteq \Delta''$. This follows from the fact
that the $\pi$-image of this cylinder lies above $y=x+\frac12$, which in turn lies inside $\Delta''$.
This implies that $x$ cannot contain $1100$ as a substring.

Our first goal will be to show that either a point avoiding $\Delta''$
    to the right is eventually of the form $01^{a_1} 0 1^{a_2} 0 \ldots $ with
    $3 \leq a_1 \leq a_2 \leq \dots$ or the point to the right is eventually
    of the form $\{10, 110\}^*$.
We will then show that this second case cannot happen.

We first claim that if $k' < k$ and $k \geq 3$ then $(x,y) \in [1 1 0 1^k\cdot 0 1^{k'} 0] \subseteq \Delta''$.
It suffices to shows this for $k' = k-1$ by monotonicity.
We see that in this case that $y > \frac{19}{24}$, and hence we wish to
    show that the points $(x,y)$ lies above the line from $(1/2, 1)$ to
    $(1/3, 19/24)$.
That is, we must show it lies above the line $y = \frac{5}{4} x + \frac{3}{8}$.
We see that
    \[x = 0.0 1^{k-1} 0 \ldots < 0.0 1^k = \frac{1}{2} - \frac{1}{2^{k+1}}, \]
and similarly that
    \[y = 0.1^k 0 1 1 \ldots > 0.1^k 0 1 1 = 1 - \frac{1}{2^k}
    + \frac{1}{2^{k+2}} + \frac{1}{2^{k+3}} = 1 - \frac{5}{2^{k+3}}. \]
This gives us that
 \[ \frac{5}{4} x + \frac{3}{8} < 1 - \frac{5}{2^{k+3}} < y \]
which proves the result.

We next claim that if $k' < k-1$ and $k \geq 3$ then
    $(x,y) \in [1 0 1^k\cdot 0 1^{k'} 0] \subseteq \Delta''$.
It suffices to show this for $k' = k-2$ by monotonicity.
We see that in this case that $y > \frac{19}{24}$, and hence we wish to
    show that the points $(x,y)$ lies above the line from $(1/2, 1)$ to
    $(1/3, 19/24)$.
That is, we must show it lies above the line $y = \frac{5}{4} x + \frac{3}{8}$.
We see that
    \[x = 0.0 1^{k-2} 0 \ldots < 0.0 1^{k-1} = \frac{1}{2} - \frac{1}{2^{k}}, \]
and similarly that
    \[y = 0.1^k 0 1  \ldots > 0.1^k 0 1 = 1 - \frac{1}{2^k} + \frac{1}{2^{k+2}}
    = 1 - \frac{2}{2^{k+2}}. \]
This gives us that
 \[ \frac{5}{4} x + \frac{3}{8} < 1 - \frac{5}{2^{k+2}}
    < 1 - \frac{2}{2^{k+2}} < y \]
which proves the result.

This shows us that if we ever have $0 1^k 0 1^{k'}$ with both
    $k, k' \geq 3$ then the sequence to the right is eventually of the form
    $0 1^{a_1} 0 1^{a_2} 0 \ldots$ with $3 \leq a_i \leq a_{i+1}$.
The dimension of the projection of the set of such points is 0.

We next claim that $(x,y) \in [11 0 11 \cdot 0 1 0] \subseteq \Delta''$.
We see that in this case that $y > \frac{19}{24}$, and hence we wish to
    show that the points $(x,y)$ lies above the line from $(1/2, 1)$ to
    $(1/3, 19/24)$.
That is, we must show it lies above the line $y = \frac{5}{4} x + \frac{3}{8}$.
We see that
    \[x = 0.0 1 0 \ldots \leq 0.0 1 1 = \frac{3}{8}, \]
and similarly that
    \[y = 0.11 0 11  \ldots \geq \frac{27}{32}.\]
This gives us that
 \[ \frac{5}{4} x + \frac{3}{8} \leq \frac{27}{32} \leq y \]
which proves the result.

This shows us that if we have $01^k0$ with $k \geq 3$ then either
    the sequence is eventually of the form $0 1^{a_1} 0 1^{a_2} 0 \ldots$ with
    $3 \leq a_i \leq a_{i+1}$, or the sequence to the right is eventually
    of the form $111 0 (11 0)^\infty$.
Thus the projection of the set of points that contain $01^k0$ for some $k \geq 3$ has dimension $0$.

Lastly we claim that $(x,y) \in [1 0 1 0 11 \cdot 0 1 0 ] \subseteq \Delta''$ or
    has a tail of the above mentioned form.
We see that in this case that $y > \frac{19}{24}$, and hence we wish to
    show that the points $(x,y)$ lies above the line from $(1/2, 1)$ to
    $(1/3, 19/24)$.
That is, we must show it lies above the line $y = \frac{5}{4} x + \frac{3}{8}$.
We may assume at this point that there are no occurrences of $111$ after
    this subsequence, otherwise we are in the above mentioned form.
That is, we must show it lies above the line $y = x + \frac{11}{24}$.
We see that
    \[x = 0.0 1 0 \ldots \leq 0.0 1 0 (110)^\infty = \frac{5}{14}, \]
and similarly that
    \[y = 0.11 0 1 0 1 \ldots \geq \frac{53}{64}.\]
This gives us that
 \[ \frac{5}{4} x + \frac{3}{8} \leq  \frac{23}{28} < \frac{53}{64} \leq y \]
which proves the result.

\medskip\noindent
{\bf Projection of $\Delta'$ on the $x$-axis.}

\medskip\noindent
To prove that this projection has zero dimension,
we first show that the set of $x$ in the projection that contain a substring $111$ has
zero dimension. We then show that the set of $x$ in the projection that contain $000$ has zero dimension.
Lastly, we show that the set of $x$ in the projection that do not contain the substrings $000$ or $111$
    has zero dimension.

Let $(x,y) \in \J(\Delta')$.
The first observation that we make is that, again, $[11\cdot 00] \subset \Delta'$.
The next thing we observe is that we cannot have the
    $1^{k+2} 0 1^k 0$ or $111 0 1^{k+1} 0 1^{k} 0$ as substrings, for all
    $k \geq 1$. Namely, $[1^{k+2} \cdot 0 1^k 0] \subset \Delta'$ and
    $[111 0 1^{k+1} \cdot 0 1^k 0] \subset \Delta'$. The first claim follows
    from the fact this cylinder lies above $y=x+\frac12$; let us prove the second one.
    Note that the line connecting $\left(\frac5{24},\frac13\right)$ and $\left(\frac12,1\right)$
    is $y=\frac{8x+3}7$. Then for the rectangle $\pi([111 0 1^{k+1} \cdot 0 1^k 0])$ we have
\begin{align*}
x&\le \frac12-\frac1{2^{k+2}},\\
y&\ge 1-\frac1{2^{k+1}}+\frac1{2^{k+3}}+\frac1{2^{k+4}}+\frac1{2^{k+5}},
\end{align*}
whence $y>(8x+3)/7$, which is a direct check.

This implies that if we have a $1^k0$ as a substring, for $k \geq 3$,
    then it must be followed by a $1$ (as $1100$ is a forbidden substring).
Thus, the substring will look like $1^k01$.
Next, by the first observation, this must be followed by additional 1s, as
    $1^{k} 0 1^{k-2} 0$ is forbidden.
Thus, the substring will look like $1^k01^{k-1}$.
After this the sequence will look like either $1^k01^{k-1}0$ or $1^k01^k$.
If we are in the second case, then our sequence is of the form
    $1^k 0 1^{k_2} 0 1^{k_3} 0 \dots$ where $3 \leq k \leq k_2
    \leq k_3 \leq \dots$ with the possibility that one of these $k_i$ is
    infinite.
The projection of the points from this second case has zero dimension.

Thus, we can assume our word is of the form $1^k 0 1^{k-1} 0$.
If $k-1 \geq 3$ we can again repeat these argument, hence without loss of
    generality assume that $k = 3$.
Hence the substring which is not forbidden has to be of the form $111 0 11 0$.

We next claim that $111 (0 11)^n 010$ and $111 (011)^n 0111 0110$ are forbidden
substrings for all $n \geq 1$. Namely, $[111(011)^n \cdot 0 1 0] \subset \Delta'$ and
$[111(011)^n 0111\cdot 0110] \subset \Delta'$. Let us prove the first one;
we have $x\le \frac38$ and $y\ge \frac67+\frac1{49}\cdot 8^{-n}>(8x+3)/7$. The second
claim is proved in a similar way as are the rest of such claims. We leave the details
to the interested reader.

This proves that the substrings will be of the form $111 (011)^\infty$ or
$111 (011)^n 01^{a_1} 0 1^{a_2} 0 1^{a_3} \dots$ with
 $3 \leq a_1 \leq a_2 \leq a_3 \leq \dots$.
The projection of this set, again, is of zero dimension.

This shows that the set of $x$ in the projection containing a $111$ as a substring has
zero dimension.

We now consider those $x$ with a substring $10^k 1$ for $k \geq 1$.
We now claim that both $10^k 1 0^{k+2}$ and $1 0^{k} 1 0^{k+1} 1 00$ are
    forbidden for $k \geq 1$.
This follows as $[10^k 1\cdot 0^{k+2}] \subset \Delta'$ and
    $[10^k1 \cdot 0^{k+1} 1 00] \subset \Delta'$.
Thus we see that we are either of the form $1 0^k 1 0^a 1$ with $a \leq k$,
   or $1 0^k 1 0^{k+1} 1 0 1$.
If we are of the first form, and $a \geq 1$ we can repeat this argument.
If instead we are of the form $1 0^k 1 0^{k+1} 1 0 1$ then
    we can again repeat this argument, with $k = 1$.
Thus we can assume that this is eventually contained within
    $\{100, 110, 10\}^*$.

Consider a word in $\{100, 110, 10\}^*$.
We have that the words of the form $110 (10)^n 110 (10)^m 100$ and
   $110 (10)^n 100 (10)^m 100$ are forbidden for all $n, m \geq 0$, since
    $[110(10)^n 11\cdot 0 (10)^m 100] \subset \Delta'$ and
    $[110(10)^n 1\cdot 00 (10)^m 100] \subset \Delta'$.
This tells us that our sequence is eventually of the form
    $110 (10)^{n_1} 110 (10)^{n_2} 110 (10)^{n_3} 110 (10)^{n_4} \dots$ or
    $110 (10)^{n_1} 100 (10)^{n_2} 110 (10)^{n_3} 100 (10)^{n_4} \dots$.

Next, $110 (10)^n 110 (10)^m 110$ is forbidden for
    $n \geq 0$ and $m \geq 2$, as well as for $n = 0, m = 1$.
This follows as $[11011\cdot 010110] \subset \Delta'$ and
    in general $[110(10)^n 11\cdot 0 (10)^m 110] \subset \Delta'$
    for $n \geq 0$ and $m \geq 2$.
This tells us that if our sequence of the form:
    $110 (10)^{n_1} 110 (10)^{n_2} 110 (10)^{n_3} \dots$,
    then eventually either $n_i = 0$ or $n_i =1$ for all $i$.
The projection of this set will have zero dimension as well.

We further observe that $110 (10)^n 100 (10)^m 110$ is forbidden for
    $n \geq 0$, and $m \geq 1$ as well as for $n = m = 0$.
This follows as $[110(10)^n 1\cdot 00 (10)^m 110] \subset \Delta'$
    for $n \geq 0$, and $m \geq 1$ as well as for $n = m = 0$.

This tells us that if we are of the form
    $110 (10)^{n_1} 100 (10)^{n_2} 110 (10)^{n_3} 100 (10)^{n_4} \dots$,
    then $n_2 = n_4 = n_6 = \dots = 0$.
Hence we are of the form $100 110 (10)^{n_1} 100 110 (10)^{n_3} 100 110 \dots$.
We observe that this is a forbidden word for all $n_1, n_3 \geq 0$.

This shows that the projection of set of $x$ in the alphabet $\{10, 110, 100\}^*$ that
    contains a $110$ is of zero dimension.

We observe that $100 (10)^n 100 (10)^m 100$ is forbidden
    for $n \geq 1$ and $m \geq 0$.
This follows as
    $[100(10)^n 1\cdot 00 (10)^m 100] \subset \Delta'$ for
    for $n \geq 1$ and $m \geq 0$.
Hence if a word contains no $110$ then it will be of the form
    $100 (10)^{n_1} 100 (10)^{n_2} 100 (10)^{n_3} \dots$, and
    hence from the comment above, will eventually be of the form
    $(100)^\infty$.
The projection of this set is clearly of zero dimension.

To summarize, if our sequence contains $01^k0$ with $k \geq 3$, then it is
    to the right
    eventually of the form $0 1^{a_1} 0 1^{a_2} 0 \ldots$ with
    $3 \leq a_i \leq a_{i+1}$, or to the right eventually of the form
    $111 0 (11 0)^\infty$.
If instead it contains no occurrences of $01^k0$ with $k \geq 3$, then
    we see that the tail to the right has to be of the form $(011)^\infty$ or
    $(01)^\infty$.
\end{proof}

Thus, Theorems~\ref{thm:Delta_0} and \ref{thm:Delta-prime} yield that if we have a convex hole
in the upper-left quadrant of the square and it is a dimension trap, then it has to contain $\Delta'$
or $\Delta''$, which are the smallest dimension traps of this kind, with area $\frac{13}{96}\approx 0.135$.

Furthermore, as a corollary to this, we see that $\mathrm{hull}(\Delta' \cup \Delta'')$ is the smallest
symmetric (with respect to $x+y=1$) dimension trap of this kind,
with area $\frac{53}{384} \approx 0.1380208333$.

\section{Cycle traps and dimension traps which contain $(1/2, 1/2)$.}
\label{sec:55}
{\ }

Let $P_1$ denote the open parallelogram with vertices $(1/3,2/3), (1/2,1), (2/3,1/3)$ and $(1/2,0)$
--- see Figure \ref{fig:P}.

This parallelogram may also be defined by the inequalities
\[
2x-1 <y <2x \quad \text{and} \quad 2-4x <y < 3-4x.
\]
These inequalities have a clear symbolic meaning, so they are easy to check.
For instance, let $x_1=1$;
then the inequality $y< 2x-1$ means $x_0x_{-1} x_{-2} \ldots \prec x_2x_3 x_4\ldots$, and $y> 3-4x$ means
$x_0x_{-1}x_{-2}\ldots \succ \overline{x_3 x_4 x_5\ldots}$

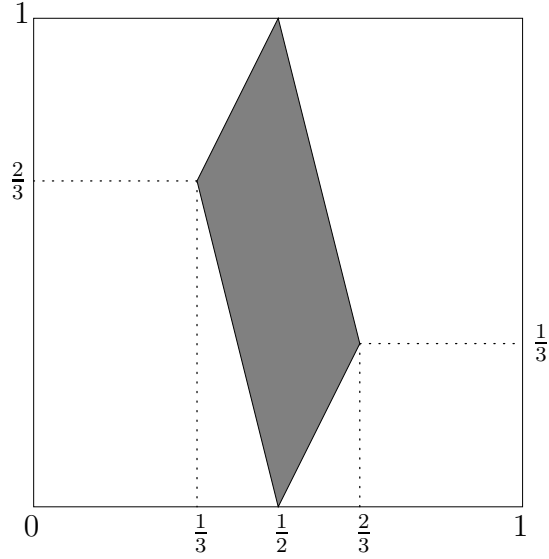
\begin{figure}
\centering
\centering \unitlength=1.3mm
    \begin{picture}(50,53)(0,0)
            \thinlines
  \path(0,0)(0,50)(50,50)(50,0)(0,0)
  \dottedline(16.7,33.35)(16.7,0)
  \dottedline(16.7,33.35)(0,33.35)
    \dottedline(33.35,16.7)(33.35,0)
      \dottedline(33.35,16.7)(50,16.7)
    \shade\path(16.7,33.35)(25,50)(33.35,16.7)(25,0)(16.7,33.35)
    \put(-1,-3){$0$}
    \put(24.5,-3.5){$\frac12$}
    \put(49,-3){$1$}
    \put(-2,49.5){$1$}
    \put(16.2,-3.5){$\frac13$}
    \put(32.85,-3.5){$\frac23$}
        \put(51,16){$\frac13$}
            \put(-2.5,32.6){$\frac23$}
          \end{picture}

          \bigskip

\caption{The cycle trap $P_1$}
\label{fig:P}
\end{figure}

Consider the Thue-Morse word $\t = t_0 t_1 t_2 \ldots = 0110\ 1001\
    1001\ 0110 \ \ldots$
    given by the recurrence $t_{2m} = t_m, t_{2m+1} = 1-t_m$ with $t_0 = 0$.
Define $\t_k = (t_0 t_1 t_2 \ldots t_{2^k-1})^\infty$.
Let $P_k$ be the hexagon with vertices
    $(1/2, 0), (\t_k, 2\t_k), (\t_k, 2-4 \t_k), (1/2, 1), (1-\t_k, 2-4 \t_k), (1-4 \t_k, 2 \t_k)$.
This hexagon may also be defined by the inequalities
\[
    2x-1 <y <2x, \quad 2-4x <y < 3-4x \quad \text{and} \quad \t_k < x < 1-\t_k
\]
We define $P_\infty = \lim_{k \to \infty} P_k$.
See Figure \ref{fig:P-infty}.

\begin{figure}[t]
\centering
\centering \unitlength=1.3mm
    \begin{picture}(50,53)(0,0)
            \thinlines
  \path(0,0)(0,50)(50,50)(50,0)(0,0)
  \dottedline(20.6,17.6)(20.6,0)
  \dottedline(20.6,41.2)(0,41.2)
    \dottedline(20.6,16.7)(0,16.7)
  \dottedline(20.6,41.2)(0,41.2)
    \dottedline(29.4,32.4)(50,32.4)
      \dottedline(29.4,8.8)(50,8.8)
      \dottedline(29.4,8.8)(29.4,0)

    \shade\path(20.6,41.2)(20.6,17.6)(25,0)(29.4,8.8)(29.4,32.4)(25,50)(20.6,41.2)
    \put(-1,-3){$0$}
    \put(-9,16){$2-4\t$}
    \put(49,-3){$1$}
    \put(-2,49.5){$1$}
    \put(20,-3.5){$\t$}
    \put(27,-3.5){$1-\t$}
    \put(-3.3,40){$2\t$}
        \put(51,32){$4\t-1$}
            \put(51,8){$1-2\t$}
          \end{picture}

          \bigskip

\caption{The hexagon $P_\infty$.}
\label{fig:P-infty}
\end{figure}
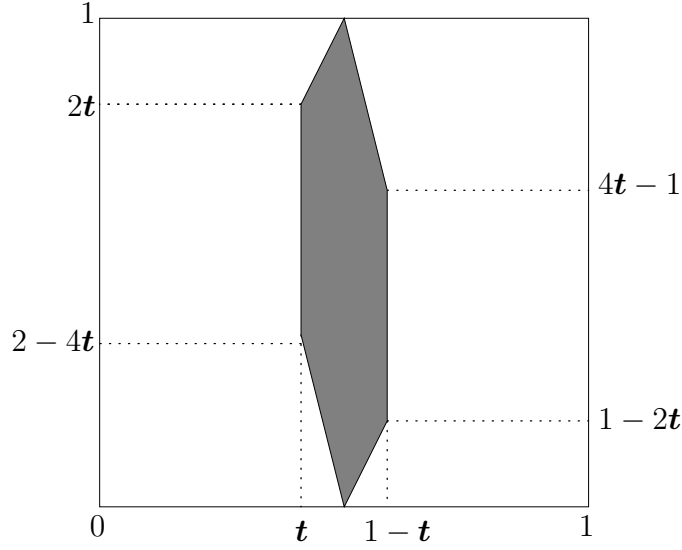

There are a number of goals to this section.
\begin{itemize}
\item Show that for each $k \geq 1$ that there are only a finite number of cycles that avoid
$\overline{P_k}$, and moreover that $P_1$ is a cycle trap.
\item Show that for each $k \geq 1$ that $P_k$ is a dimension trap.
\item Show that $P_\infty$ is a dimension trap.
\end{itemize}

\begin{lemma}\label{lem:P}
No element of $\J(P_k)$ can contain the factors
\[
\begin{array}{llll}
(a) & 1 0^\ell 1^n 0, 0 1^\ell 0^n 1 \hspace{0.5 in}& n > \ell \geq 2
    \hspace{0.5 in} & P_k, k \geq 1 \\
(b) & 1 0^\ell 1 0^n 1, 0 1^\ell 0 1^n 0 \hspace{0.5 in}& n > \ell \geq 1 \hspace{0.5 in} & P_1 \\
(c) & 1 0^\ell 1 0^n 1, 0 1^\ell 0 1^n 0 \hspace{0.5 in}& n > \ell \geq 1, n \geq 3
    \hspace{0.5 in} & P_k, k \geq 2 \\
(d) & 1^n 0^n 1^n 0 1, 0^n 1^n 0^n 1 0 \hspace{0.5 in} & n \geq 2 \hspace {0.5 in} & P_1, P_2  \\
(e) & 1^n 0^n 1^n 0 1, 0^n 1^n 0^n 1 0 \hspace{0.5 in} & n \geq 3 \hspace {0.5 in} & P_k, k \geq 3 \\
(f) & 1^n 0^n 1 0^n 1 0, 0^n 1^n 0 1^n 0 1 \hspace{0.5 in} & n \geq 2 \hspace{0.5 in} & P_1, P_2 \\
(g) & 1^n 0^n 1 0^n 1 0, 0^n 1^n 0 1^n 0 1 \hspace{0.5 in} & n \geq 3 \hspace{0.5 in} & P_k, k \geq 3
\end{array}
\]
\end{lemma}

\begin{proof}
We observe that the second word for each case follows by the symmetric $(x, y) \to (1-y, 1-x)$, so
    we will prove the first restriction only.

{\bf Case 1: Proof of (a):} \\
We observe that if $n > \ell \geq 2$ then
    $[1 0^{\ell-1} \cdot 0 1^{n} 0] \subseteq P_k$ for all $k$, from
    which the desired result follows.

{\bf Case 2: Proof of (b) and (c):} \\
We observe that if $n > \ell \geq 1$ then
    $[1 0^{\ell-1} \cdot 0 1^{n} 0] \subseteq P_1$.
Further if $n \geq 3$ then
    $[1 0^{\ell-1} \cdot 0 1^{n} 0] \subseteq P_k$ for all $k$.

{\bf Case 3: Proof of (d) and (e):} \\
We observe that if $n \geq 2$ then
$[1^n 0^{n-1} \cdot  0 1^n 0 1] \subseteq P_2 \subseteq P_1$.

If in addition, $n \geq 3$ we have that
$[1^n 0^{n-1} \cdot  0 1^n 0 1] \subseteq P_k$ for all $k \geq 3$.

{\bf Case 4: Proof of (f) and (g):} \\
We observe that if $n \geq 2$ then
    $[1^n 0^n \cdot 1 0^n 1 0] \subseteq P_2 \subseteq P_1$.

If in addition, $n \geq 3$ we have that
    $[1^n 0^n \cdot 1 0^n 1 0] \subseteq P_k$ for all $k \geq 3$.
\end{proof}

\begin{cor} {\ }
\label{cor:P_k}
\begin{enumerate}
\item
The only cycles of the baker's map which do not intersect $P_1$ are
    of the form $(10^n)^\infty, (01^n)^\infty$, $(1^n0^n)^\infty$ or
    $(0^n 1 0^n 1^n 0 1^n)^\infty$.
    All of these intersect $\overline{P_1}$.
    Hence $P_1$ is a cycle trap.
\item
The only cycles of the baker's map which do not intersect $P_2, P_3, \dots$ are
    of the form $(10^n)^\infty, (01^n)^\infty$, $(1^n0^n)^\infty$ or
    $(0^n 1 0^n 1^n 0 1^n)^\infty$
    or are contained in the Cantor set $\{01, 001, 011, 0011\}^*$.
    Cycles of the form $(10^n)^\infty, (01^n)^\infty$, $(1^n0^n)^\infty$ or
    $(0^n 1 0^n 1^n 0 1^n)^\infty$, with $n \geq 2$ intersect $\overline{P_k}$ for
    $k \geq 2$.
\end{enumerate}
\end{cor}

\begin{proof}
Consider a cycle of the form $(1^{a_1} 0^{a_2} 1^{a_3} \ldots 0^{a_k})^\infty$ with $k$ even.
We will label this as $a_1 a_2 \ldots a_k$ for convenience.
If all $a_i = 1$, then we are done.
Assume without loss of generality that $a_1 = \min_{a_i \neq 1} a_i$.

By Lemma~\ref{lem:P}~(a), $a_2 = 1$ or $1 < a_2 \leq a_1$.
Since $a_1$ is chosen minimally, we have $a_2 = 1$ or $a_2 = a_1$.
In the case $a_2 = a_1$ we proceed to show that $a_3 = 1$ or $a_3 = a_1$.
If instead we have $a_2 = 1$ then by Lemma~\ref{lem:P}~(b) and (c) we have that
    $a_3 \leq a_1$.
Again, by minimality we have $a_3 = a_1$ or $a_3 = 1$.
If both $a_2 = a_3 = 1$ and the cycle does not intersect $P_1$, then by Lemma~\ref{lem:P}~(b)
    we have $a_2 = a_3 = \dots = a_k = 1$.
By rotating this cycle, we get a contradiction with Lemma~\ref{lem:P}~(b), whence $a_3 = a_1$.

If instead both $a_2 = a_3 = 1$ and the cycle does not intersect $P_k$ for some $k \geq 2$
    then by Lemma \ref{lem:P} (c) we have either that have $a_2 = a_3 = \dots = a_k = 1$,
    or that there exists some $i$ such that $a_i = 2$.
By the minimality of $a_1$ we see then that this cycle is contained in $\{01, 001, 011, 0011\}^*$.

By continuing in this fashion, we see that for all $i$ that $a_i = 1$ or $a_i = a_1$.
In the case of $P_2, P_3, \dots$ we will assume that $a_1 \geq 3$,
    else all $a_i \in \{1, 2\}$ as required.

It follows from Lemma~\ref{lem:P}~(b) and (c) that if for some $i$ we have $a_i = a_{i+1} = 1$, then
    $a_i = 1$ for all $i$. In other words, we cannot have two consecutive $1$s.

Lemma~\ref{lem:P}~(d) and (e) implies that if we have three consecutive $a_1$ in the sequence,
    then $a_i = a_1$ for all $i$.

Finally, if $a_i = a_{i+1} = a_1$ and $a_{i+2} = 1$ then $a_{i+3} = a_1$ (as we do not
    have two consecutive $1$s).
Furthermore, we then get that $a_{i+4} = a_1$ by Lemma~\ref{lem:P}~(f) and (g).
Repeating this argument gives us the cycle $(0^{a_1} 1^{a_1} 0 1^{a_1} 0^{a_1} 1)^\infty$.

We see for $n \geq 2$ that all of these cycles intersect $\overline{P_k}$.
Finally, we observe that the cycle $(01)^\infty$ intersects $\overline{P_1}$.
\end{proof}

\begin{prop}
\label{prop:P_1 dim}
The set $P_1$ is a dimension trap.
\end{prop}

\begin{proof}
Any sequence in $\J(P_1)$ is of one of the forms
\begin{align*}
\ldots 1^{a_i} 0^{a_{i+1}} 1^{a_{i+2}} \ldots 0^{a_n} 1 0^{b_1} 1 0^{b_2} 1 0^{b_3} \ldots  \\
\ldots 1^{a_i} 0^{a_{i+1}} 1^{a_{i+2}} \ldots 1^{a_n} 0 1^{b_1} 0 1^{b_2} 0 1^{b_3} \ldots  \\
\ldots 1^{a_i} 0^{a_{i+1}} 1^{a_{i+2}} 0^{a_{i+3}} \ldots \\
\ldots 0 1^{b_1} 0 1^{b_2} 0 1^{b_3} \ldots  \\
\ldots 1 0^{b_1} 1 0^{b_2} 1 0^{b_3} \ldots
\end{align*}
where $a_i \leq a_{i+1}$ and $b_i \leq b_{i+1}$  for all $i$ and $a_n \leq b_1$.

This shows that these are eventually periodic in the left direction.
Further, using a similar argument to that in Theorem~\ref{thm:Gamma-dimtrap} and
    Lemma~\ref{lem:entropy-dim} we conclude that the dimension of this set is zero.
\end{proof}

\begin{prop}
\label{prop:P2}
The only cycle in  $\J(\overline{P_2})$ is $\left\{\left(\frac13, \frac23\right), \left(\frac23, \frac13\right)\right\}$.
\end{prop}

\begin{lemma}
\label{lem:4.5}
Let $\mathsf{B}$ be the set of sequences that are eventually periodic to the left and are
    eventually of the form $0 1^{a_1} 0 1^{a_2} 0 \ldots$, $1 0^{a_1} 1 0^{a_2} 1 0^{a_3} \ldots$
    or $0^{a_1} 1^{a_2} 0^{a_3} 1^{a_3} \ldots$ with $a_1 \leq a_2 \leq a_3 \leq \dots$ in the right
    direction.
Consider a sequence $a_i \in \{1, 2\}$ corresponding to $0^{a_1} 1^{a_2} 0^{a_3} 1^{a_4} \ldots$.
No element of $\J (P_2) \setminus \mathsf{B}$ written as $\ldots a_{-1} a_0 a_1 a_2 \ldots$
with $a_i \in \{1, 2\}$ may contain a factor
\[
\begin{array}{lll}
(a) & 1 1 2 1 \hspace{1 in}\\
(b) & 2 2 2 1 \hspace{1 in}\\
(c) & 2 2 1 2 1\hspace{1 in} \\
(d) & 1^n 2 2 1^m 2 \hspace{1 in} & n > m,\ m - \text{odd} \\
(e) & 2 1^n 22 1^m \hspace{1 in} & 2 \neq n < m, n-\text{even} \\
\end{array}
\]
\end{lemma}

\begin{proof}
We note that the set $\mathsf{B}$ is of zero dimension, whence we can ignore this set
    without loss of generality.
As we are assuming the sequence is not in $\mathsf{B}$ and in $\J(P_2)$,
    we may assume that the sequence has left or right tail (or
    possibly both) of the form $\{011, 01, 001, 0011\}^*$.

{\bf Case 1: Proof of (a):} \\
This follows by noticing that that the rectangles
    $[.10010(010)^\infty,.10010(110)^\infty] \times [.01(0)^\infty,.01(1)^\infty]$ and
    $[.10010(0)^\infty,.10010(1)^\infty] \times [.01(001)^\infty,.01(101)^\infty]$ are
    contained within $P_2$.
The symmetry $(x,y) \leftrightarrow (1-x, 1-y)$ proves the equivalent statement swapping the
    roles of $0$ and $1$ in the above statement.

{\bf Case 2: Proof of (b):} \\
This is similar to case (a) by looking at the shift $\ldots 1001\cdot 10010 \ldots$ with the
    restrictions that we are not in $\mathsf{B}$.

{\bf Case 3: Proof of (c):} \\
This is similar to case (a) by looking at the shift $\ldots 10011\cdot 01101 \ldots$ with the
    restrictions that we are not in $\mathsf{B}$.

{\bf Case 4: Proof of (d):} \\
This is similar to case (a) by looking at the shift $\ldots (10)^{(m+3)/2} 1\cdot 1 0 0 (10)^{(m-1)/2}1001 \ldots$
    with the restrictions that we are not in $\mathsf{B}$.

{\bf Case 5: Proof of (e):} \\
This is similar to case (a) by looking at the shift $\ldots 100(10)^{n/2} 1\cdot 1 0 0 (10)^{n+1} \ldots$
    with the restrictions that we are not in $\mathsf{B}$.
\end{proof}

\begin{proof}[Proof of Proposition \ref{prop:P2}]
Consider a cycle $(0^{a_1} 1^{a_2} \dots 1^{a_k})^\infty$.
We can assume without loss of generality that all of the $a_i \leq 2$,
    for otherwise they are of one of the forms specified by Corollary~\ref{cor:P_k}.
    By Lemma~\ref{lem:4.5}~(b), we cannot have more than two 2s
   in a row, otherwise we have the cycle $(0011)^\infty$ which is of a form
    specified in Corollary~\ref{cor:P_k}.
By this and Lemma~\ref{lem:4.5}~(a), if we have more than one 1,
   then it must be followed by exactly two 2s.

If all $2$s are isolated (i.e., if we do not have an occurrence of
    two $2$s in a row), then by Lemma~\ref{lem:4.5}~(a), we can have at
    most one $1$, and hence we have a word of the form $(0 11)^\infty$ or
    $(100)^\infty$, which are cycles of an admissible form.
Hence we may assume that we have one occurrence of two $2$s in a row.
By Lemma~\ref{lem:4.5}~(a) and (c) we see that all occurrences of $2$
    must come in pairs.

Hence $a_1 a_2 a_3 \dots $ looks like $2 2 1^{k_1} 2 2 1^{k_2}
    \dots 2 2 1^{k_n}$ for some choice of $k_i$.
If $k_1= 2m > 0$ is even, then by Lemma~\ref{lem:4.5}~(e) we have $k_1 \geq k_2$.
Furthermore, by Lemma~\ref{lem:4.5}~(d), $k_2$ cannot be odd.
Continuing in this manner, we have that $k_1 \geq k_2 \geq \dots \geq k_n \geq k_1$,
where all $k_i$ are even. This is of the form $(00 11 (01)^m)^\infty$ or $(1100(10)^m)^\infty$.
A similar argument can be used if $k_1$ is odd to show that this
    must be of the form $(00 11 (01)^m 0 11 00 (10)^m 1)^\infty$.
Both of these points lie on the boundary of $P_2$.
\end{proof}

\begin{lemma}
\label{lem:levels}
Let $k \geq 2$ and $(x,y) \in P_{k-1} \setminus P_k$.
Assume that the binary expansion of $x$ does not contain $000$ or $111$.
If $(x,y) \in \J(P_k)$ then $B^{2^{k-1}}(x,y) \in P_{k-1} \setminus P_k$.
Furthermore $x = .\t_{k}^\infty$ or $x = .(\overline{\t_{k}})^\infty$.
\end{lemma}

\begin{proof}
Let $(x,y) \in P_{k-1} \setminus P_k$ and $(x,y) \in \J(P_k)$.
By symmetry we may assume that $x < 1/2$.
Thus we have
\[
2-4x < y < 2x \ \ .\t_{k-1}^\infty < x \leq .\t_k^\infty.
\]
By the second inequality, $x = 0.\t_{k-1}\ldots$; in particular, $(x,y)$ is contained in the polygon $\Pi_k$ with vertices
$(x_{k-1}, 2 x_{k-1}), (x_k, 2 x_k), (x_k, 2-4 x_k), (x_{k-1}, 2-4 x_{k-1})$
where $x_{k-1} = 0.\t_{k-1}^\infty$, and similarly for $x_k$. (We have that
$\Pi_k$ is a trapezium for $k\ge3$ and a triangle for $k=2$.)
Since $x = 0.\t_{k-1} \ldots$, we see that the first $2^{k-1}$ iterates of $B$ on $\Pi_k$ are continuous.
Furthermore,
\begin{align*}
B^{2^{k-1}}(x_{k-1}, y)  &=
\begin{cases}
(x_k, 0.\t_{k-1} + \frac{y}{2^{2^k}}) & \mathrm{if}\ k\ \mathrm{odd,} \\
(x_k, 0.\overline{\t_{k-1}} + \frac{y}{2^{2^k}}) & \mathrm{if}\ k\ \mathrm{even;}
\end{cases} \\
B^{2^{k-1}}(x_{k}, y)  &=
\begin{cases}
(1-x_k, 0.\t_{k-1} + \frac{y}{2^{2^k}}) & \mathrm{if}\ k\ \mathrm{odd,} \\
(1-x_k, 0.\overline{\t_{k-1}} + \frac{y}{2^{2^k}}) & \mathrm{if}\ k\ \mathrm{even,}
\end{cases}
\end{align*}
where $y$ can take any value in $\Pi_k$.
Hence, regardless if $k$ is even or odd, we see that the images of all four corners of
    the polygon are contained within $P_{k-1}$.
Hence, as we picked $(x,y) \in \J(P_{k})$ this proves that
    $(x,y) \in P_{k-1} \setminus P_k$.
This in turn proves that $x = 0.\t_k^\infty$ as required.
\end{proof}

\begin{cor}
There are a finite number of cycles in $\J(\overline{P_k})$.
\end{cor}

\begin{proof}
Recall that $P_1$ is a cycle trap, whence $\J(\overline{P_1})$
does not contain any cycles.
If a cycle intersects some $\overline{P_{\ell-1}}$ but not $\overline{P_{\ell}}$ for
    $\ell \leq k$ we see that it is of the form $\bm t_{\ell-1}^\infty$.
Hence cycles in $\J (P_k)$ are
    of the form $\t_{\ell}^\infty$ for $2\le\ell \leq k-1$.
\end{proof}

\begin{rmk}Thus, the sequence $\{P_k\}_1^\infty$ is a route to chaos for $B$
which is completely analogous to $\{(\bm t_k^\infty, (\overline{\bm t_k})^\infty)\}_1^\infty$
for the doubling map -- see Section~\ref{sec:doubling}.
\end{rmk}

\begin{prop}
\label{prop:Pk}
The set $P_k$ is a dimension trap  for all $k$.
\end{prop}

\begin{proof}
Let $(x,y) \in \J(P_k)$.
We can assume without loss of generality that $(x,y) \not \in \J(P_1)$,
    as the set of all such points is of zero dimension.
This implies that there exists some shift of $(x,y)$ such that
    $(x',y') \in P_1$.
Assume without loss of generality that $x' < 1/2$.
Assume that $(x', y') \in P_{\ell-1} \setminus P_\ell$.
Then we see that
    $x' = 0.\t_{a_0}^{b_0} \t_{a_1}^{b_1} \dots
            \t_{a_{n-1}}^{b_{n-1}} \t_{a_n}^\infty$ where
    $\ell \leq a_0 < a_1 < a_2 \dots < a_n \leq k$.
The set of such $x'$ is clearly countable.

Either there is no first such occurrence of $(x', y')$ in which case
    the word is of the form
    \[ \t_{a_0}^{\infty} \t_{a_1}^{b_1} \dots
      \t_{a_{n-1}}^{b_{n-1}} \t_{a_n}^\infty\]
    or there is a first such occurrence.
If there is a first such occurrence, we see that the word is
    as described by Proposition \ref{prop:P_1 dim}, which has a projection
    on to the $y$-axis of dimension $0$.

Combining these gives that $P_k$ is a dimension trap.
\end{proof}

\begin{thm}
The set $P_\infty$ is a dimension trap.
\end{thm}

\begin{proof}
If a sequence is in $\J(P_k)$, then the claim follows from Proposition~\ref{prop:Pk}.
If a sequence is in $\J(P_\infty)\setminus \J(P_k)$ for all $k\ge1$, then
it has to be a concatenation of two types of blocks: of the form
    $1^{k_i} 0^{k_{i+1}}$, where $k_i\ge k_{i+1}$, or of the form $\t_{k_i}$ for some $k_i \geq 2$ with
    $k_i \geq k_{i-1}$.
This corresponds to a point on the boundary of $P_\infty$.

We see that a block of the first kind cannot appear after a block of the second kind.
If we have blocks of the first kind,
    some of them end with $1^{j_1} 0 1^{j_2} 0\ldots$ with $j_i\ge j_{i+1}$.
So, we have $2n+1 = k_i + k_{i+1} +\dots +$ (possibly) $(j_1+1) + (j_2+1) +\dots$. In other words, we again have an
    ordered partition of $2n+1$, and the result follows from
    the Hardy-Ramanujan formula and Lemma~\ref{lem:entropy-dim}, similarly to the proof of Theorem~\ref{thm:Gamma-dimtrap}.
If we have blocks of the second kind, then we have
    $n = 2^{k_i} + 2^{k_{i+1}} +\dots +$. Consequently, we again have a subset of
    ordered partition of $n$, and the result follows again from
    the Hardy-Ramanujan formula and Lemma~\ref{lem:entropy-dim}.
\end{proof}

\begin{prop}\label{prop:14gon}
Any dimension trap contained in $P_\infty$ must contain the vertices $(\bm t, \bm t),
    (1-\bm t, \bm t),
    (1-\bm t, 1-\bm t)$ and
    $(\bm t, 1-\bm t)$.
\end{prop}

\begin{proof}
Consider the Cantor set $\{\t_k, \overline{\t_k}\}^*$.
We see that the point $(x,y)$, $x > 1/2$ in the orbit of this Cantor set with
    minimal $x$ value has $x = \overline{\t_k} \t_k^\infty$.
Similarly, the point $(x,y)$, $x < 1/2$ in the orbit of this Cantor set with
    maximal $x$ value has $x = \t_k \overline{\t_k}^\infty$.
The possible $y$ values associated to these $x$ values are
    $\t_k \overline{\t_k}^\infty$ or $\overline{\t_k} \t_k^\infty$.
The result follows by taking $k \to \infty$.
\end{proof}

Using the techniques in Section \ref{sec:comp} we can show that if $P \subset P_\infty$ is a
    convex dimension trap, then $\mathcal L(P) \geq 0.11924$.
Here, for comparison, $\mathcal L(P_\infty) \approx 0.12911$.

\section{Searching for asymmetric dimension traps containing $(1/2, 1/2)$.}
\label{ssec:dim trap55b}

Let $r \in (0, 1/2] \cap \mathbb{Q}$ and define $a = .\zmax(r)^\infty$ and
    $b = .\omin(r)^\infty$.
From \cite{HS14} and the discussion in Section \ref{sec:doubling}  we
    know that $\dimh(\J^+(a,b)) = 0$, and moreover that $\J^+(a,b)$
    contains only a finite number of cycles.
The simplest example of this is for $r = 1/2$.
In this case $a = .\zmax(1/2)^\infty = 1/3$ and $b = .\omin(1/2)^\infty = 2/3$,
    and $\J(1/3, 2/3)$ contains only the fix points $\{0, 1\}$ and the
    2-cycle $\{1/3, 2/3\}$.
In Section \ref{sec:55} we studied the analogous construction, there labelled
    $P_1$, of the more general open polygons $P_r$ with vertices given by
\[ (1/2, 0), (a, 2a), (1/2, 1), (b, 2b-1). \]
Here again, $a = .\zmax(r)^\infty$ and $b = .\omin(r)^\infty$.
That is, $P_1$ of Section~\ref{sec:55} is $P_{1/2}$ under this new
    notation.
In this section we consider these more general $P_r$ for other rational
    numbers $r$, and show that the results analogous to those of the doubling
    map unfortunately do not hold.

\begin{nota}
Let $r = \frac{c}{d} < \frac{1}{2}$.
We define $r_0 = \frac{c_0}{d_0}$ as the unique Farey neighbour\footnote{That is,
a fraction with $|cd_0-c_0d|=1$.} of $r$
    such that $r < r_0$ and $d_0 < d$.

We define $r_k = \frac{c_k}{d_k} = \frac{c_{k-1} + c}{d_{k-1} + d}$ as
    the $k^{th}$ Farey neighbour of $r$.
\end{nota}

\begin{lemma}
\label{lem:facts}
Let $r_k$ be defined as above.
\begin{enumerate}
\item $\lim_{k \to \infty} r_k = r$.
\label{part:1}
\item $\zmax(r_k) = (\zmax(r))^{k-i} \omin(r_i) = \zmax(r_i) (\zmax(r))^{k-i}$
      for $i = 0, 1, \dots k$.
\label{part:2}
\item $\zmax(r)^\infty \prec \dots \prec \zmax(r_2)^\infty
        \prec \zmax(r_1)^\infty \prec \zmax(r_0)^\infty$.
\label{part:3}
\item $\omin(r)^\infty \prec \dots \prec \omin(r_2)^\infty
        \prec \omin(r_1)^\infty \prec \omin(r_0)^\infty$.
\label{part:4}
\end{enumerate}
\end{lemma}

\begin{example}
Consider as an example $r = \frac{3}{10}$.
One can easily check that $r_0 = \frac{1}{3}$.
This gives $r_1 = \frac{1 + 3}{3 + 10} = \frac{4}{13}$ and more
    generally, $r_k = \frac{1 + 3 k}{3 + 10 k}$.

In this case we see that
    $\zmax(r) = 0100100100$ and $\zmax(r_0)= 010$,
    $\zmax(r_1) = 0100100100100$, and in general,
    $\zmax(r_k) = (0100100100)^{k-1} 100$.
The other results can be easily verified for this example.
\end{example}

The proof follows from \cite[Lemma~5.1]{GS15} and is
    left as an exercise to the reader.

\begin{thm}
Let $r < \frac{1}{2}$.  Then there exists a $k$ such that the Cantor set
    $\{\zmax(r), \zmax(r_k)\}^*$ is disjoint from $P_r$.
\end{thm}

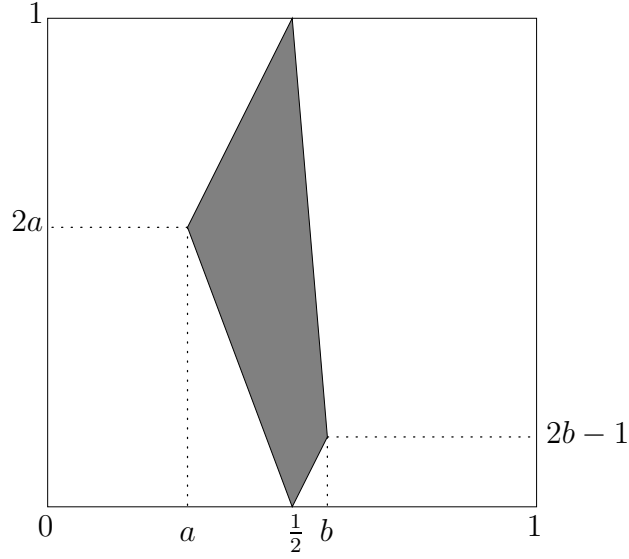
\begin{figure}
\centering
\centering \unitlength=1.3mm
    \begin{picture}(50,53)(0,0)
            \thinlines
  \path(0,0)(0,50)(50,50)(50,0)(0,0)
  \dottedline(14.3,28.6)(14.3,0)
  \dottedline(14.3,28.6)(0,28.6)
    \dottedline(28.6,7.15)(28.6,0)
      \dottedline(28.6,7.15)(50,7.15)
    \shade\path(14.3,28.6)(25,50)(28.6,7.15)(25,0)(14.3,28.6)
    \put(-1,-3){$0$}
    \put(24.5,-3.5){$\frac12$}
    \put(49,-3){$1$}
    \put(-2,49.5){$1$}
    \put(13.5,-3.5){$a$}
    \put(27.85,-3.5){$b$}
        \put(51,6.6){$2b-1$}
            \put(-3.7,28.2){$2a$}
          \end{picture}

          \bigskip

\caption{The quadrilateral $P_{1/3}$ -- {\bf not} a dimension trap.
Here $a=2/7\sim(010)^\infty, b=4/7\sim(100)^\infty$.}
\label{fig:P-onethirds}
\end{figure}

\begin{proof}
Let $x = .x_1 x_2 x_3 \ldots$ and $y = .y_1 y_2 y_3 \ldots$ be a
    shift of a point in this Cantor set.
We see that if $x_1 = 1$, then by Lemma~\ref{lem:facts} part \eqref{part:4}
    we have $b \leq x$, with equality only if $x = (\omin(r))^\infty$.
Hence the point avoids $P_r$ as required.

If $x_1 = 0$ and $y_1 = 0$, we have $x, y \leq (\zmax(r_k))^\infty$.
As the point $(a,a)$ lies outside of $P_r$, it is a bounded distance away
    from $P_r$.
This, combined with the fact that $\zmax(r_k)\to a$ as $k \to \infty$, allows
    us to choose $k$ sufficiently large so that $(x,y)$ is bounded away
    from $P_r$.

Finally, assume that $x_1 = 0$ and $y_1 = 1$.
Let $m < n$ such that $\zmax(r)$ has length~$n$ and $\zmax(r_1)$ has length~$m$.
Write
\begin{align*}
\zmax(r) & = w_1 w_2 \ldots w_n \\
\zmax(r_1) & = v_1 v_2 \ldots v_m \\
           & = w_1 w_2 \ldots w_n w_1 w_2 \ldots w_{m-n+1}
\end{align*}
This follows from Lemma~\ref{lem:facts}. (It is worth noting that $m-n < n$.)
Note that
    $\zmax(r_k) = \underbrace{\zmax(r) \ldots \zmax(r)}_{k-1} \zmax(r_1)$
    from Lemma~\ref{lem:facts}.
By noticing that $\zmax(r)$ and $\zmax(r_1)$ starts and ends with $0$,
    we see that we are in one of three cases.
\begin{enumerate}
\item $x = .w_s w_{s+1} \ldots w_n w_1 w_2 \ldots w_n \ldots$ with $s \neq 1$,
    \label{case:1}
\item $x = .w_s w_{s+1} \ldots w_n v_1 v_2 \ldots v_m \ldots$ with $s \neq 1$,
    \label{case:2}
\item $x = .v_s v_{s+1} \ldots v_m w_1 w_2 \ldots w_n \ldots$ with $s \neq 1$.
    \label{case:3}
\end{enumerate}

{\bf Case (\ref{case:1}) and (\ref{case:2}).}
By construction, we have that
\begin{align*}
w_s w_{s+1} \ldots w_n v_1 v_2 \ldots v_{s-1}
& = w_s w_{s+1} \ldots w_n w_1 w_2 \ldots w_{s-1}  \\
& \prec w_1 w_2 \ldots w_n
\end{align*}
and hence $x < a$.

{\bf Case (\ref{case:3}).}
First assume that $s > n$.
Then
\begin{align*}
x & = .v_s v_{s+1} \ldots v_m w_1 w_2 \ldots w_{n} \ldots \\
  & = .w_{s-n} w_{s-n+1} \ldots w_{m-n+1} w_1 w_2 \ldots w_{n} \ldots
\end{align*}

We clearly cannot have
   $w_{s-n} w_{s-n+1} \ldots w_{m-n+1} \succ w_1 w_2 \ldots w_{m-s+2}$.
If $$w_{s-n} w_{s-n+1} \ldots w_{m-n+1} \prec w_1 w_2 \ldots w_{m-s+2},$$
then we are done.
If instead $w_{s-n} w_{s-n+1} \ldots w_{m-n+1} = w_1 w_2 \ldots w_{m-s+2}$, then
$w_{m-s+3} = 1$ by \cite[Lemma~5.1]{GS15}
    and $w_1 = 0$, whence
    $$w_{s-n} w_{s-n+1} \ldots w_{m-n+1} w_1 \prec
    w_1 w_2 \ldots w_{m-s+2} w_{m-s+3}.$$

Next assume that $s \leq n$. We have
\begin{align*}
x & = .v_s v_{s+1} \ldots v_m \ldots \\
  & = .w_s w_{s+1} \ldots w_n w_1 w_2 \ldots w_{m-n+1} \ldots \\
\end{align*}
Similar to the above, we clearly cannot have
   $w_{s} w_{s+1} \ldots w_{n} \succ w_1 w_2 \ldots w_{n-s+1}$.
If $w_{s} w_{s+1} \ldots w_{n} \prec w_1 w_2 \ldots w_{n-s+1}$, then
    we are done.
If instead we have $w_{s} w_{s+1} \ldots w_{n} = w_1 w_2 \ldots w_{n-s+1}$, then
$w_{n-s+2} = 1$ and $w_1 = 0$, whence
    $w_{s} w_{s+1} \ldots w_{n} w_1 \prec w_1 w_2 \ldots w_{n-s+1} w_{n-s+2}$,
    as required.
\end{proof}

\section{Lower bound for the size of a dimension trap}
\label{sec:comp}

Let $\T$ denote the set of all convex open dimension traps.
We define $\delta = \inf_{T \in \T} \mathcal L(T)$.
Note that $\Delta, \Delta'$ and $P_\infty$ are all in $\T$, and
    have area $\frac{13}{96}, \frac{13}{96}$ and $0.129106155\dots$
    respectively. Therefore, $\delta \leq 0.129106155\dots$.
For the precise definition of $\Delta$ and $\Delta'$ see Section \ref{sec:01}
     and for $P_\infty$ see Section \ref{sec:55}.
The goal of this section is to find a good lower bound for $\delta$.

By Remark~\ref{rmk:required points}, a dimension trap must contain one of
      $\{(0, 1/2^n), (1/2^n, 0)\}$ for some $n \geq 1$, and must also contain
      one of $\{(1, 1-1/2^n), (1-1/2^n, 1)\}$ for some $n \geq 1$.
These are our starting ``polygons'' (although they are one-dimensional).

Let $a$ and $b$ be two finite words. We will assume throughout this section
that $.a^\infty \neq .b^\infty$. Define the Cantor sets
\begin{align*}
\C(a,b,n) & = \{a b^n, a b^{n+1}\}^*.
\end{align*}
We further define the limit of this sequence of Cantor sets as follows:
\begin{align*}
\C(a,b) & = \lim_{n\to\infty} \C(a,b,n) \\
         & = \mathcal{O}(\ldots bbb \cdot a bbb \ldots),
\end{align*}
where $\mathcal{O}(\ldots y_3 y_2 y_1 \cdot x_1 x_2 x_3 \ldots)$ is the orbit
    of this point under $B$.

We see that if $T \in \T$, then for all $a, b, n$ we have
    $\C(a,b,n) \cap T \neq \emptyset$.
Furthermore, for all $\epsilon > 0$ and all $a$ and $b$ we have that
    $\C(a,b) \cap (T+B_\epsilon) \neq \emptyset$, where, as above,
     $B_\epsilon$ is the disc centred at the origin of radius $\epsilon$.
This allows us a way to estimate $\delta$ from below.

For some finite collection of $(a,b)$, consider the family $\T_0$ of
    convex polygons such that they contain at least one point from each $\C(a,b)$.
We see that if $T \in \T$ then there exists a $T_0 \in \T_0$ such
    that $T_0 \subset T$.
This gives us that $\delta \geq \inf_{T_0 \in \T_0} \mathcal L (T_0)$.

Here we take advantage of the fact that the forward and backward orbits of
    $\ldots bbb \cdot a bbb \ldots$ tend to the orbits of
    $\ldots bbb \cdot bbb \ldots$.
Thus, for any $\epsilon > 0$ there are only a finite number of points in
    orbit $\mathcal{O}(\ldots bbb \cdot a bbb \ldots)$ that are not within
    $\epsilon$ of the orbit of $\ldots bbb \cdot bbb \ldots$.

We will first look at a small example.
Instead of looking at the full set $\T$, we will consider a restricted
    subset $\T' \subset \T$.
In particular, we will look at the subset of $\T'$ where for each
    $T' \in \T'$ we have $(1/2, 1), (1/2, 0) \in T'$ and $T'$ is closed
    under the symmetry $(x,y) \leftrightarrow (1-x, 1-y)$.
We will wish to find a set $\T_0'$ such for all $T' \in \T'$ with the above
    property either $\mathcal L(T') \geq 0.13$ or there exists a $T_0' \in \T_0'$ with $T_0' \subset T'$.
Clearly, we may assume that all $T_0' \in \T_0'$ satisfy the restriction that
   $(1/2, 1), (1/2, 0) \in T_0'$ and $T_0'$ is closed
    under the symmetry $(x,y) \leftrightarrow (1-x, 1-y)$.
By considering the infimum of $T_0'$  in $\T_0'$ we
   will show that if $T'$ is a dimension trap with the above restrictions,
   then $\mathcal L(T') \geq 293/2688 - 4\epsilon \approx 0.1090029761$,
    where $\epsilon = 10^{-10}$.
We see that $P_\infty$ is an example of such a $T'$.
Thus, the infimum of this subset is bounded above by $0.129106155\dots$.
This is the reason why our search for all polygons with $\mathcal L(T') \leq 0.13$
    above is not unreasonable.

Next, one can check that
\[
\C(0, 10) = \left\{ \left( \frac{1}{3}, \frac{2\cdot 4^n-1}{3\cdot 4^n}\right), \left(\frac{2}{3},
\frac{2\cdot 4^n-1}{6\cdot 4^n}\right), \left(\frac{2\cdot 4^n-1}{3\cdot 4^n}, \frac{1}{3}\right),
\left(\frac{2\cdot 4^n-1}{6\cdot 4^n}, \frac{2}{3}\right)\right\}_{n\geq 0}.
\]
When taking into account the symmetry, and the fact that we are searching for $T'$ with
$\mathcal L(T') < 0.13$, we see that
    only the point $(7/12, 1/3)$ is utilized.
This gives the polygon with vertices $(5/12, 2/3)$, $(1/2, 0)$, $(7/12, 1/3)$ and $(1/2, 1)$
    with area $1/12 < 0.13$.

We next consider the limit Cantor set
\begin{align*}
\C(01, 10) & =  \left\{ (5/6, 1/6), \left(\frac{1}{3}, \frac{4^{n+1} + 3}{6 \cdot 4^n}\right),
\left(\frac{2}{3}, \frac{4^n +3}{3\cdot 4^n}\right),  \right. \\
           & \left. \left(\frac{2 \cdot 4^n - 3}{3 \cdot 4^n}, 1/3\right),
           \left(\frac{2 \cdot 4^n -3}{6 \cdot 4^n}, 2/3\right) \right\}_{n \geq 1}.
\end{align*}
When taking into account the symmetry and the fact that we are searching for $T'$ with $\mathcal L(T') < 0.13$, we see that
     the points $(29/48, 1/3)$ and $(5/12, 1/3)$.
In the first case, this results in the polygon with vertices
    $(19/48, 2/3)$, $(1/2, 0)$, $(29/48, 1/3)$ and $(1/2, 1)$ with area $5/48 < 0.13$.
and the second this results in the polygon with vertices
    $(5/12, 2/3)$, $(5/12, 1/3)$, $(1/2, 0)$, $(7/12, 1/3)$, $(7/12, 2/3)$ and $(1/2, 1)$ with area $1/9 <  0.13$.

We repeat this process again with these two polygons and with
    $\C(0, 001)$, $\C(0, 011)$, $\C(00, 011)$, $\C(001, 010)$, $\C(001,101)$, $\C(001, 110)$ and $\C(011, 100)$.
(In fact we do this with all $a$ and $b$ where $|a|, |b| \leq 3$, but these were the only Cantor sets that
    contributed points outside of the polygons at this point of the algorithm.)
We take advantage here of the observation that if we have $T_1', T_2' \in \T_0'$ such that $T_1' \subset T_2'$, then we
    can remove $T_2'$ from $\T_0'$ and the results still hold.
Even with this, and only looking at those $T_0'$ with $\mathcal L(T_0') < 0.13$, this
    results in $13$ polygons, with areas ranging between $293/2688 \approx .1090029762$ to $ 49/384 \approx .1276041667$.

From this we can conclude that if $T$ is a dimension trap, $(1/2, 0), (1/2, 1) \in T$ and $T$ is closed under the symmetry
    $(x,y) \leftrightarrow (1-x, 1-y)$, then $\mathcal L(T) \geq 293/2688 - 4\epsilon \approx .1090029758$.
This result can be strengthened by taking longer $a$ and $b$, as given in Table~\ref{tab:delta}.

Here the $4 \epsilon$ arises because we are not using {\em all} points in $\C(a,b)$, only all of those points
    that are more than $\epsilon$ away from one of the limit points, and one point within $\epsilon$ of one of the limit points.
If we consider all polygons in this restricted set, then it is possible that all of the vertices should be replaced by
    vertices that are $\epsilon$ closer to the center of the polygon.
This would result in a polygon that is at most $4 \epsilon$ smaller (as the polygons are convex, and contained within $[0,1]^2$).

We summarize the lower bounds given by such computations in Table~\ref{tab:delta}.

\begin{table}[h]
\begin{center}
\begin{tabular}{llllllll}
Contains points      & Symmetry    & length   & \# of   & Lower     & Upper     & Bound      & $\epsilon$ \\
                     &             & of words & polygon & bound     & Bound     & to search  &            \\
                     &             &          & found   &           &           &            &            \\ \hline
$(0, 1/2), (1/2, 1)$ &  mirror     & 8        & 2       & $0.13793$ & $0.13803$ & $0.1381$   & $10^{-10}$ \\ 
$(0, 1/2), (1/2, 1)$ &  none       & 8        & 6       & $0.13532$ & $0.13542$ & $0.1355$   & $10^{-10}$ \\ 
$(1/2, 0), (1/2, 1)$ &  rotational & 8        & 2       & $0.11891$ & $0.12911$ & $0.119$    & $10^{-10}$ \\ 
$(1/2, 0), (1/2, 1)$ &  none       & 8        & 1253    & $0.11802$ & $0.12911$ & $0.1182$   & $10^{-10}$ \\ 
\end{tabular}
\caption{Upper and lower bounds on restricted subsets of $\T'$}
\label{tab:delta}
\end{center}
\end{table}

\begin{rmk}
These proofs are computational in nature, and we summarize some details of these computations in Table \ref{tab:delta}.
An interesting observation to be made is that if we assume $T$ is a convex dimension trap satisfying the symmetry
    $(x,y) \leftrightarrow (1-x, 1-y)$ and $\mathcal L(T) \leq 0.13$, then by considering only words up to length~4,
    we have that $T$ must contain the points $(0,1/2)$ and $(1, 1/2)$ or it
    contains the points $(1/2, 0)$ and $(1/2,1)$.
We note that $P_\infty$ is a convex dimension trap satisfying this property, hence the bound of $0.13$ is reasonable.
All other $T$ are eliminated by this process by having too large a measure.
See Figure~\ref{fig:art1} for the set of polygons in $\mathcal{T}_0$
    with the restriction that they are closed under the symmetry
    $(x,y) \leftrightarrow (1-x, 1-y)$ and $\mathcal L(T) \leq 0.13$,
    which avoid $\mathcal{C}(a, b)$ for $|a|, |b|$ less than $1$, $2$,
    $3$ or $4$.

Similarly, if we assume $T$ is a convex dimension trap satisfying the symmetry
    $(x,y) \leftrightarrow (1-y, 1-x)$ and $\mathcal L(T) \leq 0.1381$, then by considering only words up to length
    4 we have that $T$ must contain the points $(0,1/2)$ and $(1/2, 1)$ or it
    contains the points $(1/2, 0)$ and $(1, 1/2)$.
We again see that $\mathrm{hull}(\Delta' \cup \Delta'')$ satisfies this restriction, hence it is a reasonable
    region to search.
More over, by length $7$ the resulting $T_0$ in $\mathcal{T}_0$ essentially
    look like $\mathrm{hull}(\Delta' \cup \Delta'')$.
See Figure~\ref{fig:art2}.

Lastly, if we assume $T$ is a convex dimension trap with $\mathcal L(T) \leq 0.1292$, with no restrictions on the symmetry,
    nor any restrictions on containing particular points, then by considering only words up to length
    $7$, we have that $T$ must one of $(0,1/2)$ or $(1/2, 0)$ and one of $(1, 1/2)$ or $(1/2, 1)$.
We see that $P_\infty$ is an example of such a dimension trap, hence this is a reasonable region to search.
That is, such a search is equivalent to that done in row 2 or 4 of Table~\ref{tab:delta}.

These observations help justify the restricted cases that are looked at in
    Sections~\ref{sec:01} and \ref{sec:55}.

In Table~\ref{tab:delta}, we will call the symmetry $(x,y) \leftrightarrow (1-y, 1-x)$ a {\em mirror} symmetry, and
    $(x,y) \leftrightarrow (1-x, 1-y)$ a {\em rotational} symmetry.
\end{rmk}

\begin{figure}
\includegraphics[width=160pt,height=200pt,angle=270]{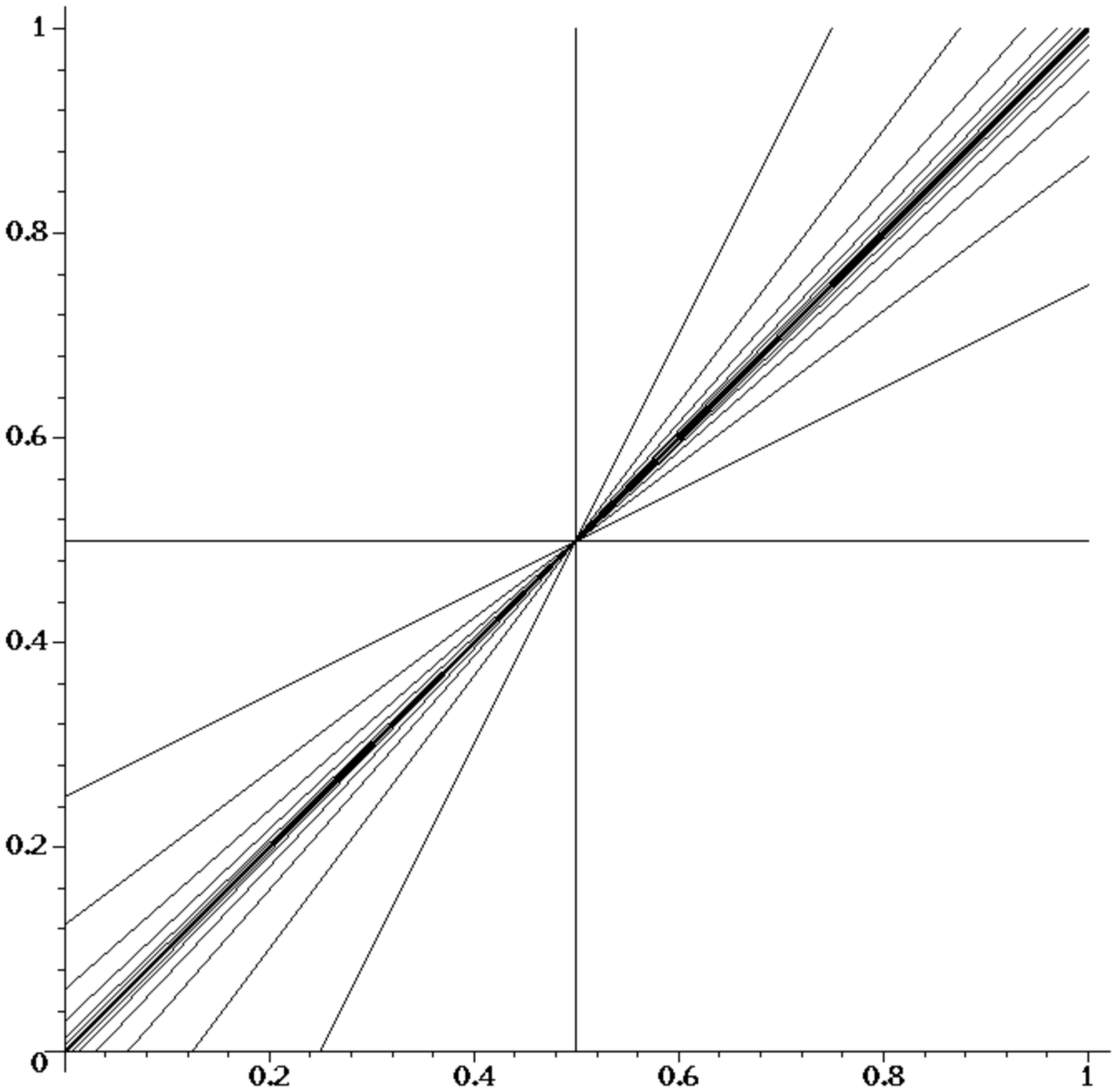}
\includegraphics[width=160pt,height=200pt,angle=270]{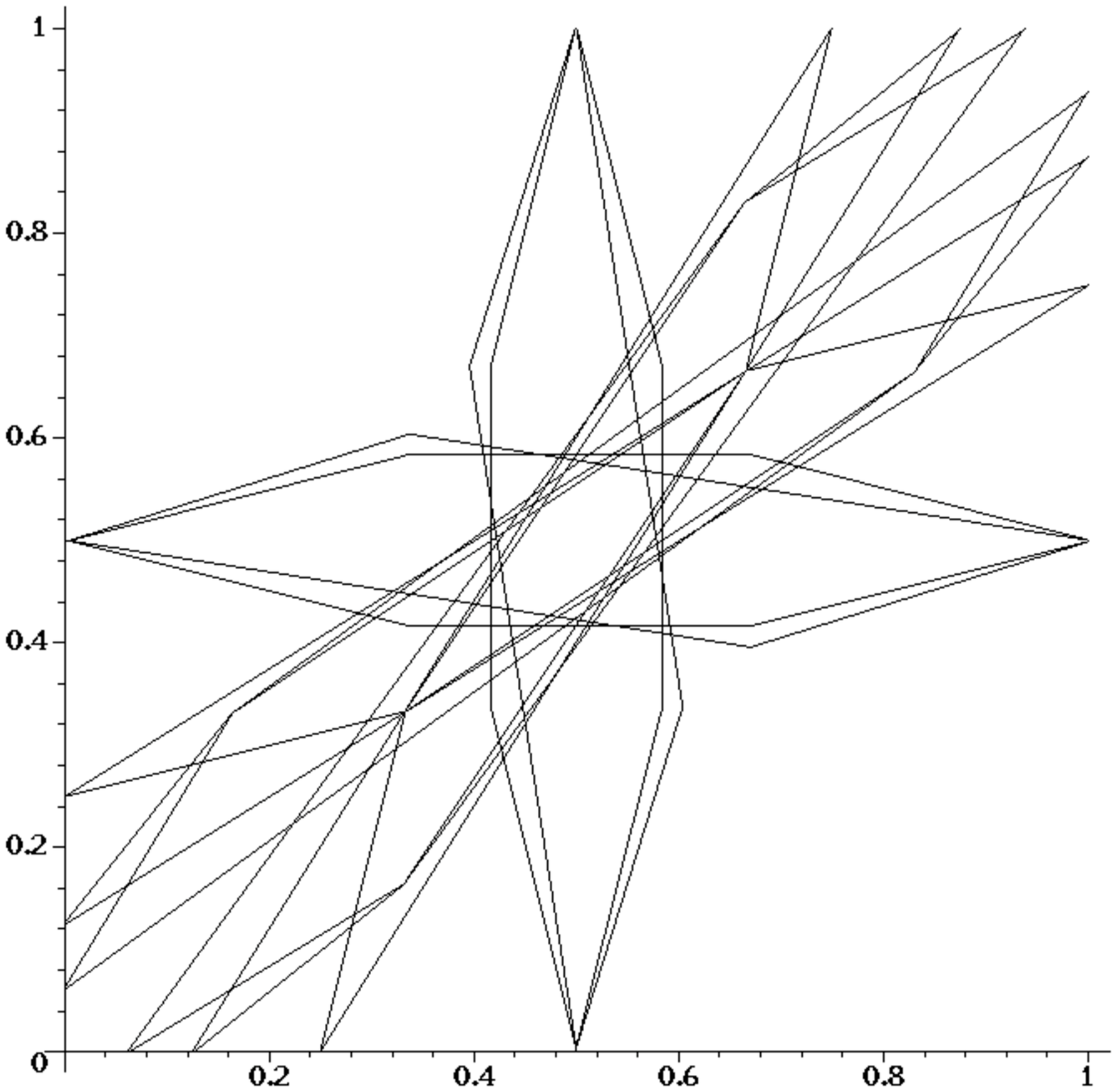} \\
\includegraphics[width=160pt,height=200pt,angle=270]{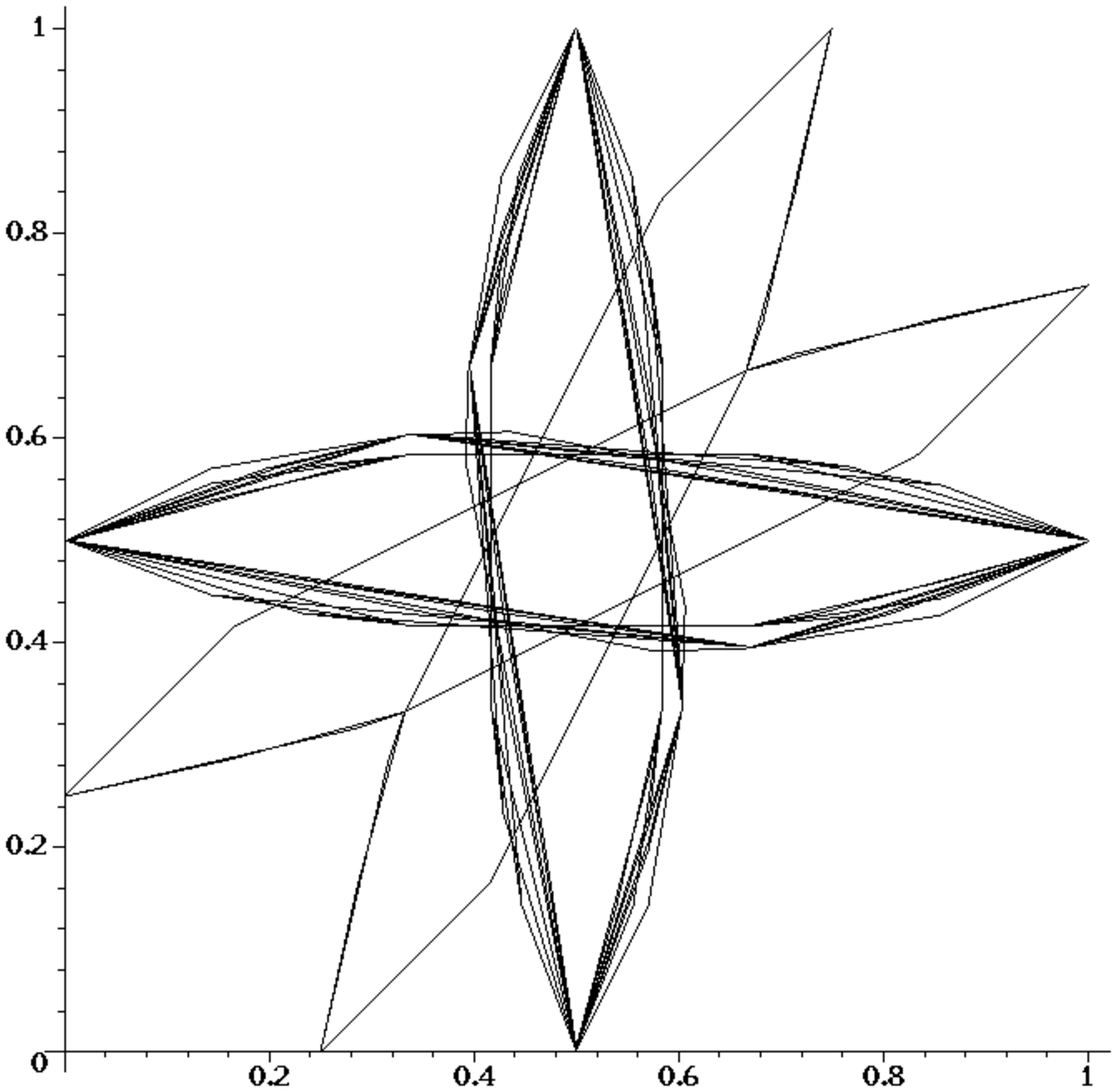}
\includegraphics[width=160pt,height=200pt,angle=270]{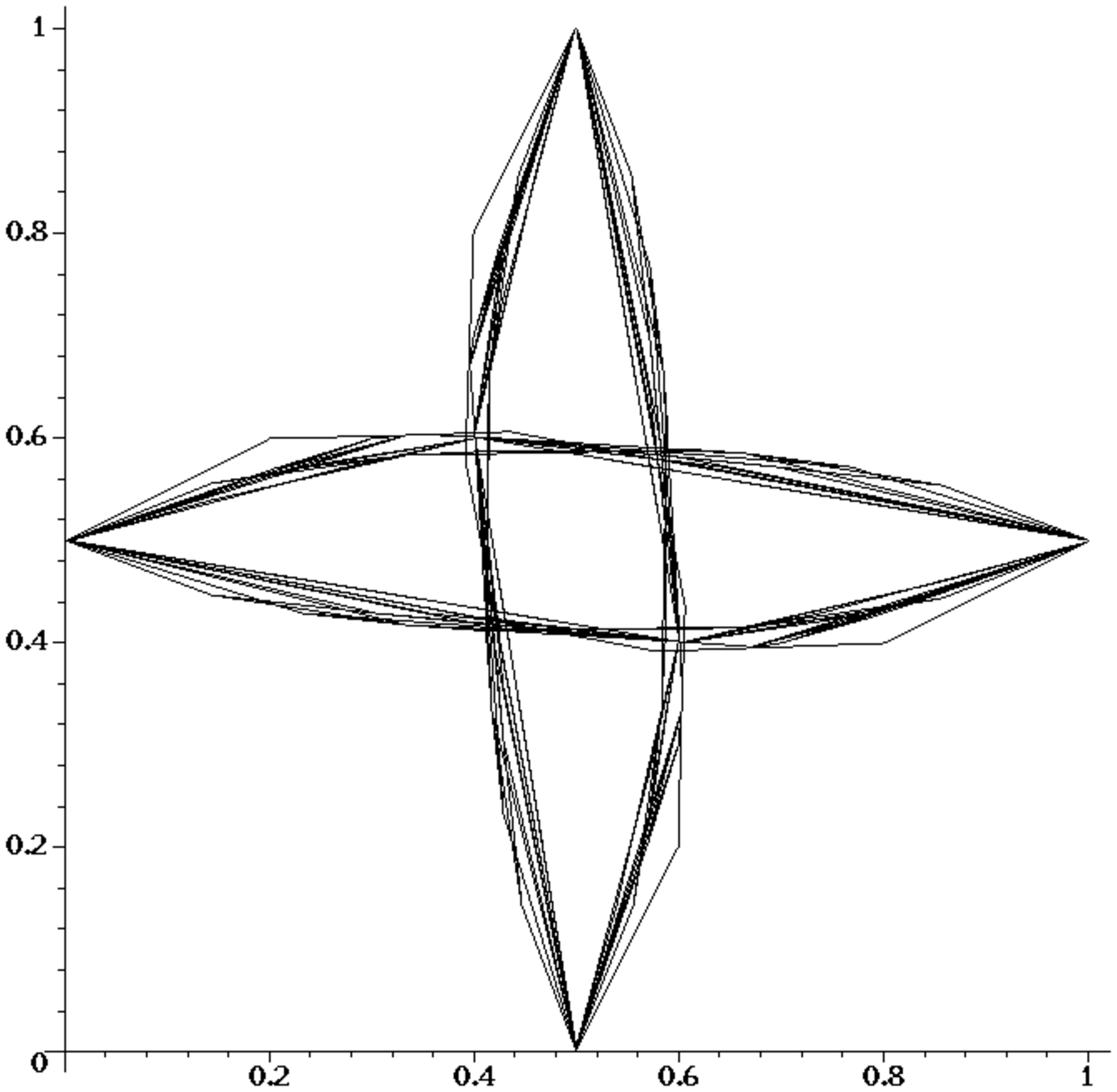}
\caption{Rotational Symmetry -- Using length 1, 2, 3 and 4 Cantor sets}
\label{fig:art1}
\end{figure}

\begin{figure}
\includegraphics[width=160pt,height=200pt,angle=270]{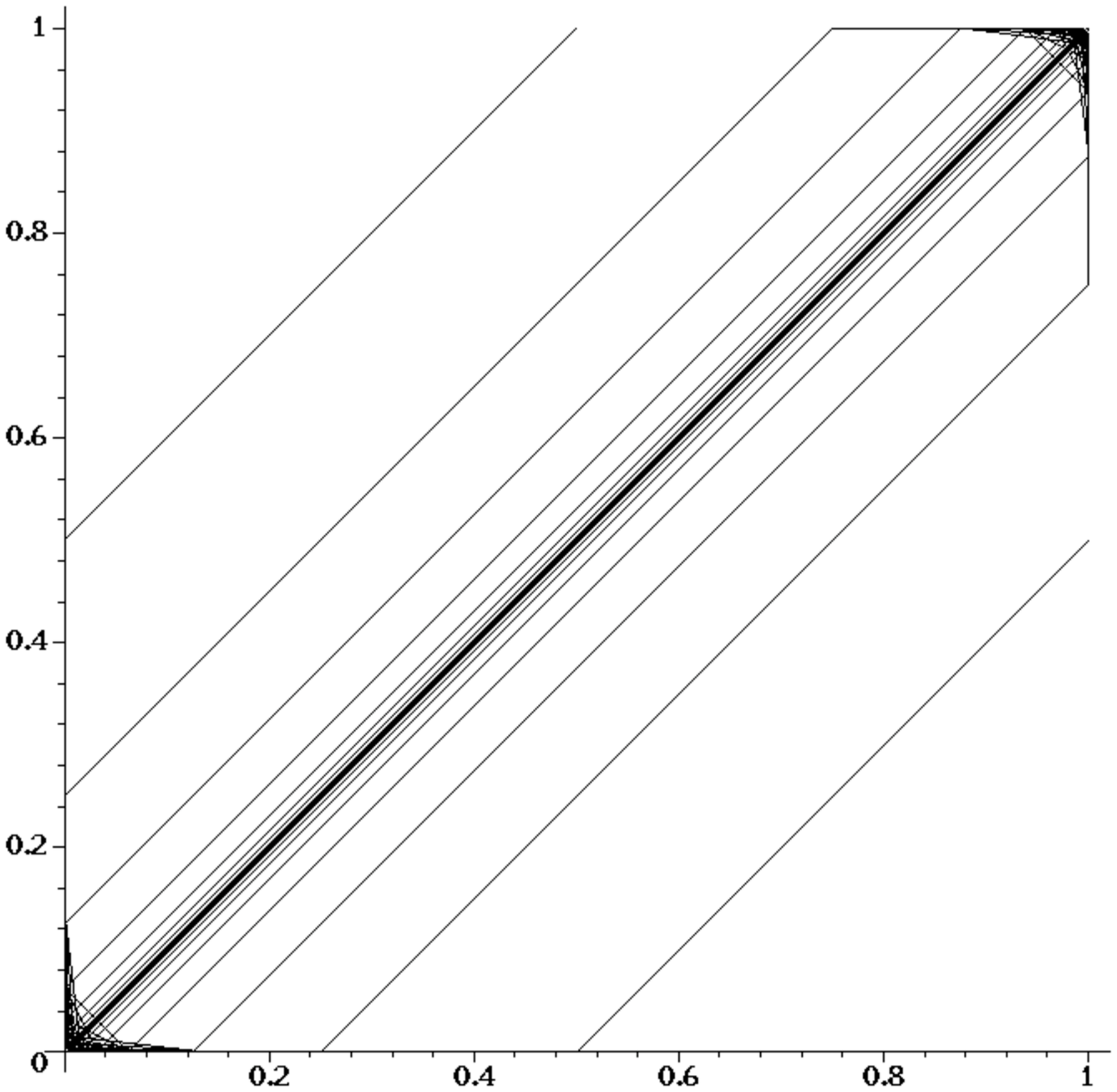}
\includegraphics[width=160pt,height=200pt,angle=270]{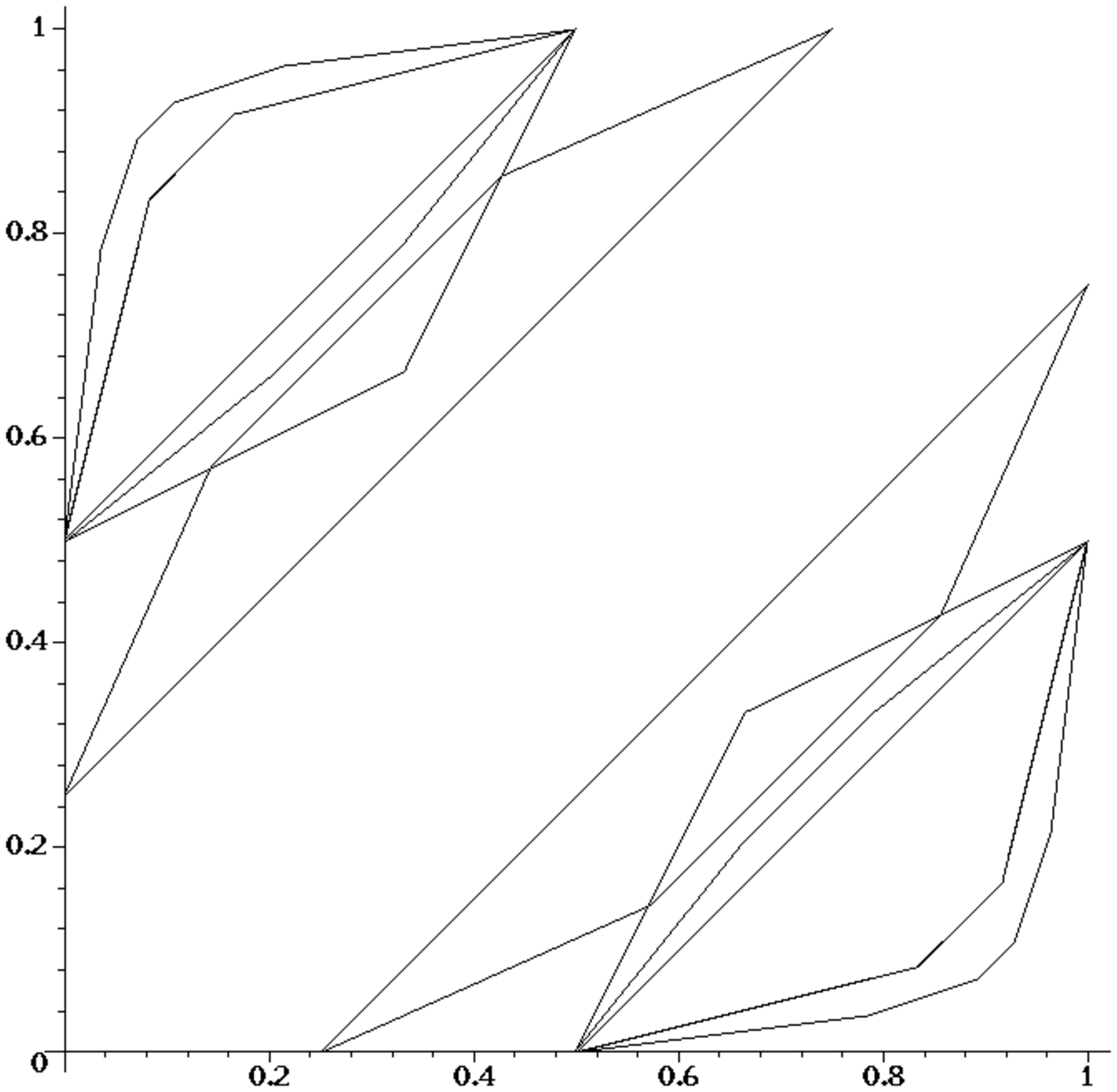}
\includegraphics[width=160pt,height=200pt,angle=270]{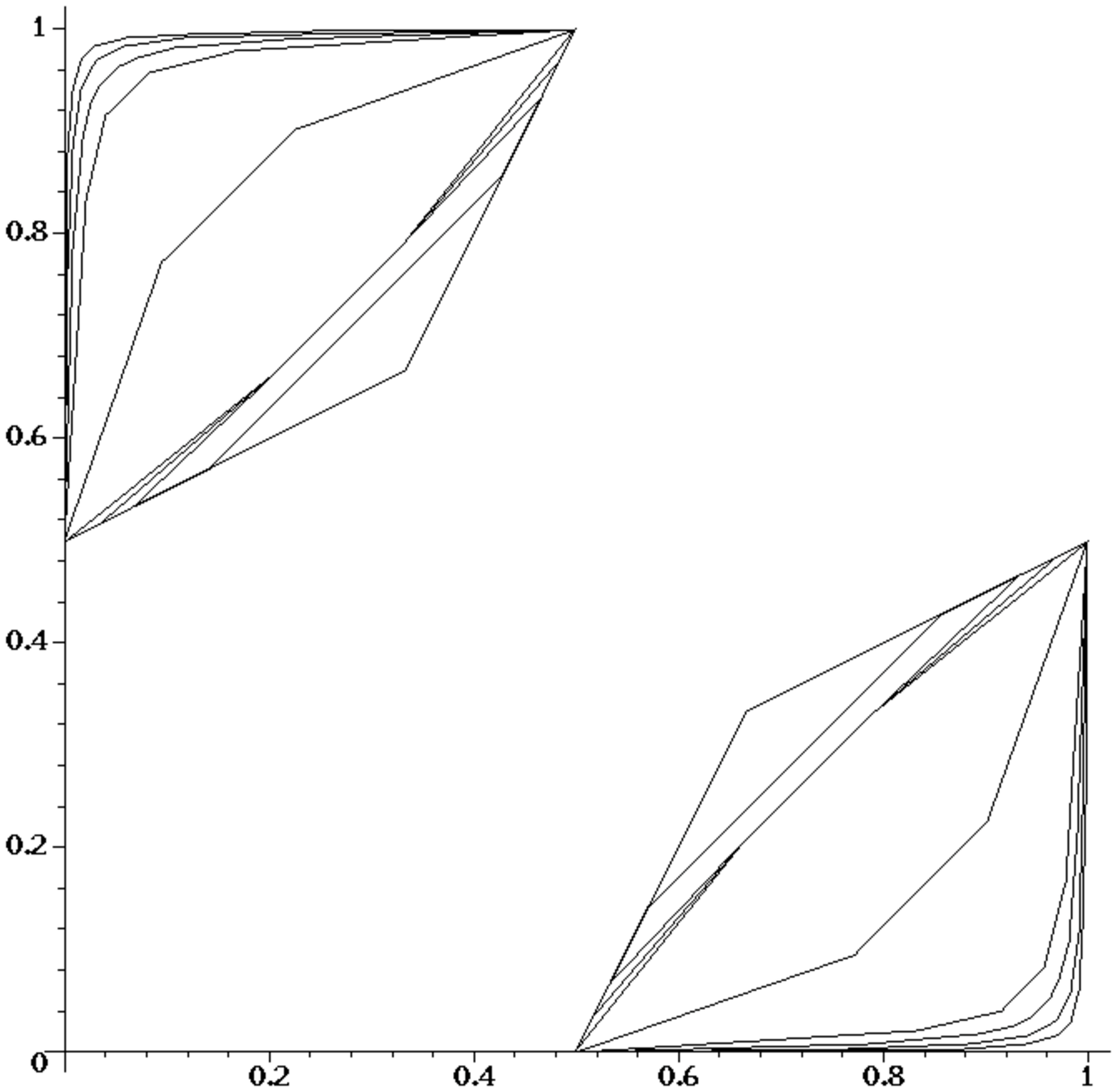}
\includegraphics[width=160pt,height=200pt,angle=270]{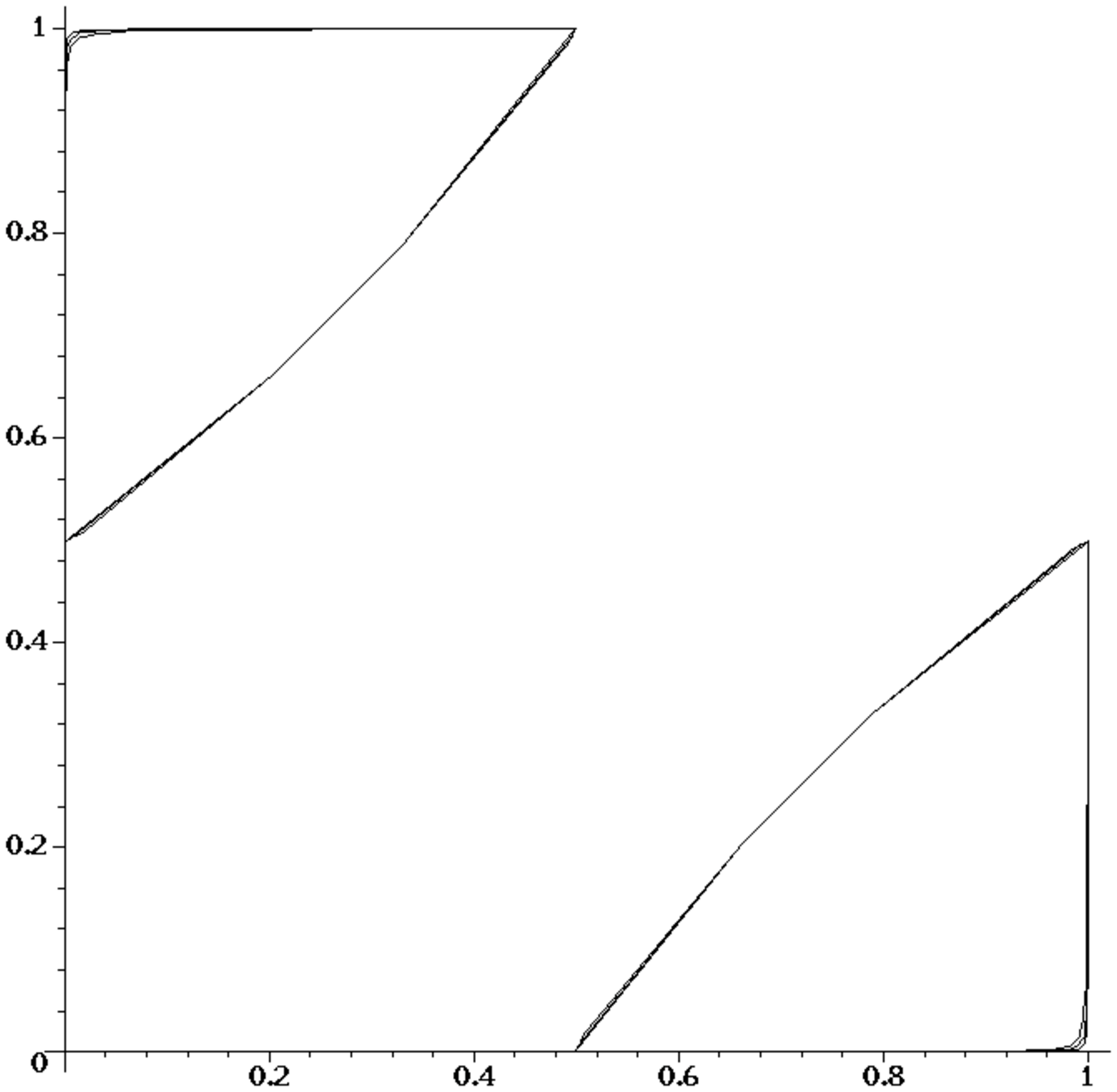}
\caption{Rotational Symmetry -- Using up to lengths 1, 3, 5 and 7 Cantor sets}
\label{fig:art2}
\end{figure}

Summing up, we obtain the following unconditional result.

\begin{thm}\label{thm:delta}
If a convex hole $H$ has area less than $0.11802$, then $\dimh  \J(H)>0$.
\end{thm}

We note that this is not optimal.
The smallest convex holes $H$ that we know with this property are
    $\Delta, \Delta'$ and $P_\infty$ which have area larger than
    $0.12916155$.
Hence it may be possible to improve this number.

\section{Final remarks and open problems}

It is worth noting that Theorem~\ref{thm:An} and the first three
    claims in Section~\ref{sec:01}
    have been proved in the first author's thesis \cite{Lyn-thesis}.
    
Theorems~\ref{thm:epsilon} and \ref{thm:An} have been generalized
in the second and third authors' recent paper \cite{HS17}. The results
are stated for subshifts; there are also applications to algebraic
toral automorphisms of a certain type.     

Some questions that arise out of this research include
\begin{enumerate}
\item Is it true that each cycle trap for $B$ is a dimension trap as well? It is
known that there exist minimal subshifts (so, in particular, they do not
have any periodic orbits) with positive topological entropy, which suggests
that the answer is probably no. See also the discussion on MathOverFlow \cite{MO-per}.

\item Our approach to finding optimal asymmetric dimension traps
from Section~\ref{ssec:dim trap55b} has been a direct attempt to fuse \cite{GS15} and Section~\ref{sec:55}
-- and
it did not work. Is there a natural family of asymmetric dimension traps leading to a generalization
of $P_\infty$?

\item  Put
\[
\delta(B) =\inf\left\{\mathcal L(H) : H\ \text{is convex, and}\ \dimh  \J(H)=0\right\}.
\]
We have shown that $\delta>0.11802$. What is the actual value of $\delta$?

\item When considering $\delta(B)$ we restrict our search of a single
    convex hole $H$.  We know that if we have no restrictions at all
    that the infimum will be $0$, as shown in Theorem \ref{thm:epsilon}.
    What other restrictions on the search space would result in interesting
    bounds?  For example, as suggested by the referee, what if
    $H$ was a union of $m$ disjoint convex sets for some fixed $m$?

\end{enumerate}

Some possible directions for future related research include:
\begin{enumerate}
\item A hole $H$ is called {\em supercritical} if $\J(H)$ is countable, and $\J(H')$
has positive dimension for any $H'$ whose closure is in $H$. For the doubling map $T$ there
exist non-trivial supercritical holes -- see \cite{Sid14}. For instance, let $\mathsf f=010010100100101001010\dots$ denote
the Fibonacci word defined as $\lim_{n\to\infty} \mathsf f_n$ given by
$\mathsf f_0=0, \mathsf f_1=1$ and $\mathsf f_{n+1}=\mathsf f_n \mathsf f_{n-1}$ for $n\ge0$.
Then $H=(.01\mathsf f, .10\mathsf f)$ is supercritical for $T$. A more general
construction involves Sturmian words. Describing supercritical holes
for the baker's map looks like an interesting problem.

\item One possible generalization of this paper would be the Fibonacci
automorphism $F:\mathbb T^2\to\mathbb T^2$ given by the formula $F(x,y)=(x+y, y)\bmod\mathbb Z^2$.
Put $\mathfrak F=\{(u_n)\in\{0,1\}^{\mathbb Z}: (u_n,u_{n+1})\neq(1,1)\ \text{for all}\ n\in\mathbb Z\}$.
Then $(\mathfrak F, \sigma)$ (equipped with the unique -- Markov -- measure of maximal entropy) is
known to be metrically isomorphic to $(\mathbb T^2, F)$ with the Haar/Lebesgue measure.

It was shown in \cite{SV} that there exists an arithmetic map $\varphi:\mathfrak F\to\mathbb T^2$ which conjugates the shift
$\sigma$ and $F$ and has an expression which is very similar to $\pi$ given by (\ref{eq:pi}). Namely,
\[
\varphi((u_n)_{-\infty}^\infty)=\left(\sum_{n=-\infty}^\infty u_n \frac{\tau^{-n}}{\sqrt5},
\sum_{n=-\infty}^\infty u_n \frac{\tau^{-n-1}}{\sqrt5}\right)\bmod\mathbb Z^2,
\]
where $\tau=(1+\sqrt5)/2$, i.e., the golden ratio. Given that the theory of critical holes for the $\beta$-transformation $[0,1)\to[0,1)$
given by $T_\beta x=\beta x\bmod1$ is similar to that
for the doubling case if $\beta$ is the golden ratio (see \cite{Lyn, Lyn-thesis}), this suggests
that most results of Sections~\ref{sec:01} and \ref{sec:55} could be probably transferred to the Fibonacci automorphism.
(The holes in question should be probably chosen to be geodesically convex.) In particular,
determining a good lower bound for $\delta(F)$ looks like an interesting question.

\item Other kinds of Pisot toral automorphisms ($=$ algebraic automorphisms
of $\mathbb T^m$ given by matrices whose characteristic polynomial is irreducible and has a Pisot root)
have been studied in
\cite{SV-2D} and \cite{Sid-Pisot} (including higher dimensions). Various results
concerning $\beta$-transformations with holes --
with $\beta>2$ \cite{Ag} and $\beta<2$ \cite{Lyn, Lyn-thesis} --
could be possibly used to build a consistent theory.
\end{enumerate}

\section{Appendix}
\label{sec:app}

In this section we consider a disconnected hole, for which the survivor set has
a really nice structure.

Put
\[
\mathcal H=\left\{(x,y)\in X : |x-y|>\frac12\right\}
\]
(see Figure~\ref{fig:b}). Let $\mathcal B$ denote the set of bi-infinite balanced words. (That is, the
words in $\{0,1\}^{\mathbb Z}$ whose every factor is balanced.)
Let $\pi$ be as defined in equation \eqref{eq:pi}.

\begin{figure}
\centering
\centering \unitlength=1.3mm
    \begin{picture}(50,53)(0,0)
            \thinlines
  \path(0,0)(0,50)(50,50)(50,0)(0,0)
    \shade\path(0,25)(0,50)(25,50)(0,25)
    \shade\path(25,0)(50,0)(50,25)(25,0)
    \put(-1,-3){$0$}
    \put(24.5,-3.5){$\frac12$}
    \put(-2,24.5){$\frac12$}
    \put(49,-3){$1$}
    \put(-2,49.5){$1$}
          \end{picture}

          \bigskip

\caption{The hole $\mathcal H$}
\label{fig:b}
\end{figure}
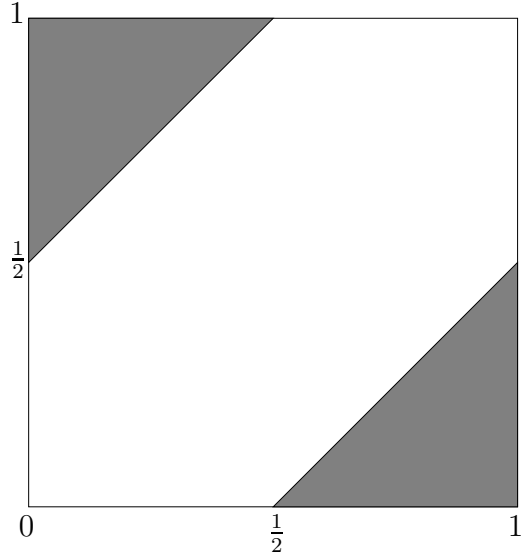

\begin{thm}
With $\mathcal{H}$ and $\mathcal{B}$ as defined above:
\begin{enumerate}
\item We have
\[
\J(\mathcal H)=\pi(\mathcal B).
\]
\item Furthermore, if $u\in\mathcal B$ is aperiodic, then for any $\varepsilon>0$ there exists
$n\in\mathbb Z$ such that $B^n(\pi(u))$ is at a distance less than $\varepsilon$ from
the boundary of $\mathcal H$, so $\mathcal H$ is in some sense optimal with this
property.
\end{enumerate}
\end{thm}
\begin{proof}\footnote{The proof is a result of a collective effort on MathOverFlow \cite{MO}, and we are
grateful to Ale De Luca and a number of anonymous colleagues who have contributed to it.} (1) Assume
first that $u\in\mathcal B$. Then for any factor $w$ of $u$ we have that $\widetilde w:=REV(w)$ is a factor of $u$ as well
(see \cite{Loth}). Assume without loss of generality that $u_0=1, u_{-1}=0$; then we cannot have
$u = \dots 0\widetilde w 0\cdot 1 w 1\dots$, otherwise $u$ would not be balanced. Hence
\[
\left|\sum_{k=0}^\infty u_{k+N}2^{-k-1}-\sum_{k=1}^\infty u_{-k+N}2^{-k}\right|<\frac12
\]
for all $N\in\mathbb Z$, which means that the orbit of $u$ stays outside $\mathcal H$.

Now assume $u\notin\mathcal B$; then there exists a $w$ such that $w=\widetilde w$ and that $0w0$ and $1w1$ are factors of $u$
(see \cite[Proposition~2.1.3]{Loth}). If $w$ is of minimal length~$n$, say, then $u$ has at most $k+1$ factors of each length $k\le n+1$
(\cite[Proposition~2.1.2]{Loth}). Moreover, $\{0w0,1w1\}$ are the only unbalanced factors of $u$ of length $\le n+2$.

Consequently, the set of factors of length $n+2$ is contained in $\{0w0,1w1\}\ \cup \
\mathrm{CYCLE}(w01)\ \cup\ \mathrm{CYCLE}(w10)$, where $\mathrm{CYCLE}(x)$ stands for
$x$ and all of its cyclic shifts. Also, every factor of $u$ of length $n+1$ except $0w$ and $1w$ always occurs
followed by the same letter.

From this we can derive that there exists an $m\ge0$ such that every occurrence of $0w0$ is followed by $(1w0)^m$ and every
occurrence of $1w1$ is followed by $(0w1)^m$. If such $m$ is chosen to be maximal, then either
\[
0w0(1w0)^m 1w1(0w1)^m=0(w01)^m w01w(10w)^m 1
\]
or
\[
1w1(0w1)^m 0w0(1w0)^m=1(w10)^m w10w(01w)^m 0
\]
is a factor of $u$. Hence the orbit of $u$ falls into $\mathcal H$, since
$[0(w01)^m w0\cdot 1w(10w)^m 1]\subset\mathcal H$
as well as its mirror image.

\medskip\noindent (2) Since $u$ is aperiodic and balanced, it is Sturmian (see \cite[Theorem~2.1.5]{Loth}). Therefore,
by the same theorem, it is a cutting sequence with an irrational slope.

Note that if we have a line $y=ax$,
then clearly its cutting sequence in both directions is the same, i.e., they are each other's reverses.
Hence for any $n\ge1$ there exists $\delta>0$ such that if we have $y=ax+\delta$, then
$u_i=u_{1-i},\ 2\le i \le n$. Note also that if our straight line comes close
to any grid point in $\mathbb Z^2$, we must have $u_1=1, u_0=0$ or vice versa.
It follows from Diophantine approximations that any straight line with an irrational slope
comes arbitrarily close to $\mathbb Z^2$, which implies that for any $\delta$ there exists $N\in\mathbb Z$
such that
\[
\left|\sum_{k=0}^\infty u_{k+N}2^{-k-1}-\sum_{k=1}^\infty u_{-k+N}2^{-k}\right|>\frac12-\delta.
\]
\end{proof}

\end{document}